\documentclass{amsart}

\input{mint.tex}

\usepackage{hyperref}

\usepackage[normalem]{ulem}

\usepackage{amssymb,amsmath,amsthm,mathtools,color}
\usepackage{enumerate}

\newcommand{\x}{\scalebox{1.2}{$\chi$} }

\theoremstyle{plain}
\newtheorem{theorem}{Theorem}
\newtheorem{proposition}[theorem]{Proposition}
\newtheorem{lemma}[theorem]{Lemma}
\newtheorem{corollary}[theorem]{Corollary}

\newtheorem{conjecture}[theorem]{Conjecture}
\newtheorem{claim}[theorem]{Claim}

\theoremstyle{definition}
\newtheorem{definition}[theorem]{Definition}

\theoremstyle{remark}
\newtheorem{remark}[theorem]{Remark}

\newcommand{\N}{\mathbb N}
\newcommand{\Z}{\mathbb Z}
\newcommand{\R}{\mathbb R}
\newcommand{\D}{\mathbb D}
\newcommand{\C}{\mathbb C}

\newcommand{\RiemannSphere}{\widehat{\mathbb{C}}}

\newcommand{\dist}{\textup{dist}}
\newcommand{\area}{\textup{Area}}
\newcommand{\loc}{\mathrm{loc}}
\newcommand{\br}{\overline}
\newcommand{\diam}{\textup{diam}}
\newcommand{\length}{\textup{length}}

\DeclareMathOperator{\id}{id}
\DeclareMathOperator{\Mod}{\mathrm{Mod}}
\DeclareMathOperator{\inter}{\mathrm{int}}
\DeclareMathOperator{\clu}{Clu}

\renewcommand {\tilde} {\widetilde}

\numberwithin{equation}{section}
\numberwithin{theorem}{section}
\begin{document}

\title{Rigidity theorems for circle domains}

\date{\today}

\author[D.\ Ntalampekos]{Dimitrios Ntalampekos}
\address{Institute for Mathematical Sciences, Stony Brook University, Stony Brook, NY 11794, USA.}
\email{dimitrios.ntalampekos@stonybrook.edu}
\thanks{The first author was partially supported by NSF Grant DMS-1506099}
\author[M. Younsi]{Malik Younsi}
\thanks{The second author was supported by NSF Grant DMS-1758295}
\address{Department of Mathematics, University of Hawaii Manoa, Honolulu, HI 96822, USA.}
\email{malik.younsi@gmail.com}


\begin{abstract}
A circle domain $\Omega$ in the Riemann sphere is conformally rigid if every conformal map from $\Omega$ onto another circle domain is the restriction of a M\"{o}bius transformation. We show that circle domains satisfying a certain quasihyperbolic condition, which was considered by Jones and Smirnov \cite{JOS}, are conformally rigid. In particular, H\"{o}lder circle domains and John circle domains are all conformally rigid. This provides new evidence for a conjecture of He and Schramm relating rigidity and conformal removability.
\end{abstract}

\maketitle

\section{Introduction}
\label{sec1}

A domain $\Omega$ in the Riemann sphere $\RiemannSphere$ is called a \textit{circle domain} if every connected component of its boundary is either a {geometric circle} or a point. Such domains are well-known to be of significant importance in complex analysis and related areas, mainly because they are expected to represent every planar domain, up to conformal equivalence. This is known as Koebe's Kreisnormierungsproblem.

\begin{conjecture}[Koebe \cite{KOE2}]
\label{KoebeConjecture}
Any domain in $\RiemannSphere$ is conformally equivalent to a circle domain.
\end{conjecture}

Koebe himself proved Conjecture \ref{KoebeConjecture} in the case of domains with finitely many boundary components \cite{KOE1}, using a dimension argument based on Brouwer's {inva\-riance-of-domain theorem}. The following generalization is undoubtedly one of the most important advances on Koebe's conjecture.

\begin{theorem}[He--Schramm \cite{SCH1}]
\label{HeSchrammTheorem}
Any domain in $\RiemannSphere$ with at most countably many boundary components is conformally equivalent to a circle domain.
\end{theorem}

As for the uncountable case, Conjecture \ref{KoebeConjecture} remains wide open, despite some partial results by Sibner \cite{SIB2}, He and Schramm (\cite{SCH4}, \cite{SCH3}), and Herron and Koskela \cite{HK}.

In Theorem \ref{HeSchrammTheorem}, the circle domain is actually unique up to a M\"{o}bius transformation, which follows from the fact, also proved in \cite{SCH1}, that every conformal map from a circle domain with at most countably many boundary components onto another circle domain is M\"{o}bius. This motivates the following definition.

\begin{definition}
A circle domain $\Omega$ in $\RiemannSphere$ is \textit{conformally rigid} if every conformal map from $\Omega$ onto another circle domain is the restriction of a M\"{o}bius transformation.
\end{definition}

The rigidity property in the countable case was in fact crucial in the proof of Theorem \ref{HeSchrammTheorem}. He and Schramm later extended rigidity to circle domains with boundaries of $\sigma$-finite length \cite{SCH2}, which is as far as we know the best rigidity result in the literature.

As observed in \cite{SCH2}, the notion of conformal rigidity appears to be closely related to conformal removability.

\begin{definition}
A compact set $E \subset \widehat\C$ is \textit{conformally removable} if every homeomorphism of $\RiemannSphere$ that is conformal outside $E$ is actually conformal everywhere, and hence is a M\"{o}bius transformation.
\end{definition}

Examples of removable sets include sets of $\sigma$-finite length and quasicircles. On the other hand, compact sets of positive area are never conformally removable. The converse is false. In fact, there exist removable sets of Hausdorff dimension two and non-removable sets of Hausdorff dimension one, and no geometric characterization of removability is known. See \cite[Section 5]{YOU} for more information, including applications to holomorphic dynamics and to conformal welding.

If $E \subset \mathbb{C}$ is a Cantor set, then $\Omega\coloneqq \RiemannSphere \setminus E$ is a circle domain, and it follows directly from the definitions that $\Omega$ is not conformally rigid if $E$ is not removable. In particular, complements of positive-area Cantor sets are never conformally rigid. The problem of determining exactly which circle domains are rigid remains open, and a solution should provide substantial insight into Koebe's conjecture. Motivated by this, He and Schramm proposed the following characterization.

\begin{conjecture}[Rigidity conjecture \cite{SCH2}]
\label{RigidityConjecture}
A circle domain $\Omega$ is conformally rigid if and only if its boundary $\partial \Omega$ is conformally removable.
\end{conjecture}

As previously mentioned, the direct implication holds if $\Omega$ has only point boundary components. Furthermore, Conjecture \ref{RigidityConjecture} holds for circle domains with boundaries of $\sigma$-finite length, in view of the preceding remarks. It also holds if $\partial \Omega$ has positive area, since in this case $\Omega$ cannot be conformally rigid, as can be seen using quasiconformal deformation of Schottky groups. In \cite{YOU2}, the second author obtained further evidence in favor of the rigidity conjecture by proving that a circle domain is conformally rigid if and only if it is quasiconformally rigid, meaning that every quasiconformal mapping from the domain onto another circle domain is the restriction of a quasiconformal mapping of the whole sphere. In particular, rigidity of circle domains is quasiconformally invariant, which would also follow if Conjecture \ref{RigidityConjecture} were true, by the quasiconformal invariance of removability (see e.g.\ \cite[Proposition 5.3]{YOU}).

It is well-known, however, that from the point of view of removability and rigidity, considerations of Hausdorff measure and dimension are not enough, and it is rather the ``shape'' than the ``size'' of the set that matters. In this spirit, our main theorem is the following rigidity result.

\begin{theorem}
\label{mainthm1}
Let $\Omega \subset \RiemannSphere$ be a circle domain, and assume without loss of generality that $\infty \in \Omega$. Let $B(0,R)\subset \C$ be a large open ball that contains all complementary components of $\Omega$. Suppose that for a point $x_0\in B(0,R)\cap \Omega$ we have
\begin{align}\label{mainthm1:condition}
\int_{B(0,R)\cap \Omega} k(x,x_0)^2 \,dx <\infty,
\end{align}
where $k(\cdot,\cdot)$ denotes the quasihyperbolic  distance in the region $B(0,R)\cap \Omega$. Then $\Omega$ is conformally rigid.
\end{theorem}

We refer to \eqref{mainthm1:condition} as the \textit{quasihyperbolic condition}. This condition was considered by Jones and Smirnov in \cite{JOS} for the study of conformal removability. More precisely, they proved that domains satisfying \eqref{mainthm1:condition} have conformally removable boundaries. The combination of this fact with Theorem \ref{mainthm1} shows that Conjecture \ref{RigidityConjecture} holds for circle domains satisfying the quasihyperbolic condition.

The quasihyperbolic condition is satisfied for sufficiently regular domains, such as John domains and H\"{o}lder domains for instance (see \cite{SMI} for the definitions). This yields the following corollary.

\begin{corollary}
\label{maincorollary1}
H\"{o}lder circle domains are all conformally rigid. In particular, John circle domains are conformally rigid.
\end{corollary}

The proof of Theorem \ref{mainthm1} is inspired by techniques of He and Schramm \cite{SCH2} and can be briefly described as follows. Let $\Omega$ be a circle domain satisfying the quasihyperbolic condition, and let $f\colon  \Omega \to \Omega^{*}$ be a conformal map of $\Omega$ onto another circle domain $\Omega^{*}$ with $f(\infty)=\infty \in \Omega^*$. The first step is to show that $f$ extends to a homeomorphism of $\overline{\Omega}$ onto $\overline{\Omega^{*}}$. In order to do this, one first needs to rule out the possibility that $f$ maps some point boundary component to a circle, or vice versa. This is proved in \cite{SCH2} using a so-called generalized Gr\"{o}tzsch extremal length argument. However, the argument relies in a crucial way on the fact that $\partial \Omega$ intersects almost every line through the origin (and almost every circle centered at the origin) in an at most countable set. This holds provided $\partial \Omega$ has $\sigma$-finite length, as is assumed in \cite{SCH2}, but may fail under the quasihyperbolic assumption only. We circumvent this difficulty using so-called detours, which were formalized in \cite{Nt}, as well as techniques inspired from \cite{JOS}. Once $f$ has been shown to extend to a homeomorphism of $\overline{\Omega}$ onto $\overline{\Omega^*}$, one can use reflections across the boundary circles to extend $f$ to a homeomorphism of the whole sphere that conjugates the Schottky groups of $\Omega$ and $\Omega^*$. The next step is to use a modulus argument, based on the fact that $f$ is absolutely continuous ``up to the boundary" (see Proposition \ref{ac:lemma Sobolev}), to show that $f$ is quasiconformal on $\RiemannSphere$ with {quasiconformal dilatation} $K$ less than some uniform constant $K_0=K_0(\Omega)$ depending only on $\Omega$. Now, if $f$ is not M\"obius and thus $K>1$, then one can use the measurable Riemann mapping theorem to construct another quasiconformal mapping of $\RiemannSphere$ that maps $\Omega$ conformally onto another circle domain but has {quasiconformal dilatation} bigger than $K_0$, which is the maximal allowed dilatation for such a map. This contradiction shows that $K=1$ and therefore $f$ must be a M\"{o}bius transformation.

Lastly, we mention that rigidity with respect to more general classes of maps (e.g.\ quasisymmetric) was extensively studied by Bonk, Kleiner, Merenkov, Wildrick and others (\cite{BKM}, \cite{MER}, \cite{MER2}, \cite{MEW}), in the case of Schottky sets. Although circle domains and Schottky sets are quite different (the latters are not domains and do not have point boundary components), some of our techniques may apply in this other setting.

The paper is structured as follows. Section \ref{sec2} contains preliminaries on the quasihyperbolic condition and detours of paths, and Section \ref{sec3} contains topological results that will be needed for the proof of Theorem \ref{mainthm1}. In Section \ref{sec4}, we prove that boundary circles map to boundary circles, and Section \ref{sec5} contains the proof that point boundary components are mapped to point boundary components. Then, in Section \ref{sec6}, we prove the continuous extension of $f$ to the boundary of $\Omega$. Section \ref{sec7} contains the proof of the quasiconformal extension to the whole sphere. In Section \ref{sec8}, we conclude the proof of Theorem \ref{mainthm1}. Finally, in Section \ref{sec9} we discuss further remarks on Conjecture \ref{RigidityConjecture}.

\subsection*{Acknowledgments}
Part of this project was completed while the first author was visiting the University of Hawaii. He thanks the Faculty and Staff of the Department of Mathematics for their hospitality. Both authors would also like to thank Mario Bonk for several fruitful conversations and suggestions, {as well as the anonymous referee for many helpful comments and corrections.}

\section{Preliminaries and the quasihyperbolic condition}
\label{sec2}

In this section we shall prove some important properties of domains $D \subset \mathbb C$ satisfying the quasihyperbolic condition of Theorem \ref{mainthm1}, i.e.,
\begin{align}\label{quasi:condition}
\int_{D} k(x,x_0)^2 \,dx <\infty
\end{align}
for some point $x_0\in D$. We also include in Subsections \ref{ac:section} and \ref{distortion:section} a few preliminaries required for the proof of the main result. We first start with some definitions.

Let $D \subsetneq \mathbb C$ be a domain, i.e., a connected open set. For a point $x\in D$, define $\delta_D(x)\coloneqq \dist(x,\partial D)$ (using the Euclidean distance). We define the \textit{quasihyperbolic distance} of two points $x_1,x_2\in D$ by
\begin{align*}
k_D(x_1,x_2)=\inf_{\gamma} \int_\gamma \frac{1}{\delta_D(x)}\,ds,
\end{align*}
where the infimum is taken over all rectifiable paths $\gamma\subset D$ that connect $x_1$ and $x_2$; here the symbol $\gamma$ denotes also the trace of the path $\gamma$. The subscript $D$ will be omitted if the domain is implicitly understood.

\begin{remark}\label{quasi:invariant}
The quasihyperbolic distance is trivially invariant under Euclidean isometries. Namely, if $T\colon \C \to \C$ is an isometry and $D\subsetneq \C$ is a domain with $x_1,x_2\in D$, then $k_{T(D)}(T(x_1),T(x_2))= k_D(x_1,x_2)$. Also, the quasihyperbolic distance is scale invariant; in other words, if $r>0$ and $T(x)=rx$, then the above equality holds as well. Hence, the quasihyperbolic condition \eqref{quasi:condition} is invariant under translation and scaling.

Furthermore, if $T\colon D\to T(D)\subsetneq \C$ is a bi-Lipschitz map, then
$$k_{T(D)}(T(x_1),T(x_2)) \simeq k_D(x_1,x_2)$$
and this shows that condition \eqref{quasi:condition} is invariant under bi-Lipschitz maps. Here and in what follows the notation $A\simeq B$ means that there exists a constant $C>0$ such that $C^{-1}A\leq B\leq CA$ and $A\lesssim B$ means that $A\leq CB$.
\end{remark}

\vspace{1em}

A simple (i.e., injective) curve $\gamma\colon [0,1]\to D$ is called a \textit{quasihyperbolic geodesic} if for any two points $x_1,x_2\in \gamma$ we have
\begin{align*}
k(x_1,x_2)= \int_{\gamma|_{[x_1,x_2]}} \frac{1}{\delta_D(x)}\,ds,
\end{align*}
where $\gamma|_{[x_1,x_2]}$ denotes the subpath of $\gamma$ between $x_1$ and $x_2$. We allow the possibility that $\gamma$ is defined on a (half) open interval and does not have endpoints in $D$, or it even accumulates at $\partial D$. A compactness argument shows that for any two points $x_1,x_2\in D$, there exists a quasihyperbolic geodesic that connects them, see e.g.\ \cite[Theorem 2.5.14]{BBI}.

For a domain $D\subsetneq \mathbb C$ we also consider the \textit{Whitney cube decomposition} $\mathcal W(D)$, which is a collection of closed dyadic cubes $Q\subset D$ (or rather squares) with the following properties:
\begin{enumerate}
\smallskip
\item the cubes of $\mathcal W(D)$ have disjoint interiors and $\bigcup_{Q\in \mathcal W(D)}Q=D$,
\smallskip
\item $\sqrt{2} \, \ell(Q) < \dist(Q,\partial D) \leq 4 \sqrt{2} \, \ell(Q)$ for all $Q\in \mathcal W(D)$,
\smallskip
\item if $Q_1\cap Q_2\neq \emptyset$, then $1/4 \leq \ell(Q_1)/\ell(Q_2)\leq 4$, for all $Q_1,Q_2\in \mathcal W(D)$.
\end{enumerate}

Here, $\ell(Q)$ denotes the sidelength of $Q$. See \cite[Theorem 1, p.~167]{ST} for the existence of the decomposition. Note that (2) implies that $k(x_1,x_2) \leq 1$ for all $x_1,x_2$ lying in the same cube $Q$, so that in particular Whitney cubes have uniformly bounded quasihyperbolic diameter.

We fix a basepoint $x_0\in D$, and denote by $k(x_0,A)$ the quasihyperbolic distance from $x_0$ to a set $A\subset D$. For each $j\in \mathbb N$ we define
\begin{align*}
D_j\coloneqq \{ Q\in \mathcal W(D): k(x_0,Q)\leq j\},
\end{align*}
and $D_0\coloneqq \emptyset$. Note that $D_j$ contains finitely many cubes for each $j\in \N$, since all cubes of $D_j$ are contained in a compact subset of $D$; see also Remark \ref{quasi:remark number} below. Each Whitney cube $Q$ is contained in $D_j\setminus D_{j-1}$ for some unique $j\in \mathbb N$. In this case, we define $j(Q)\coloneqq j$. Also, we have $D= \bigcup_{j=1}^\infty\bigcup_{Q:j(Q)=j} Q$. Two important observations are the following:

\begin{remark}\label{quasi:remark number}
Let $\gamma$ be a quasihyperbolic geodesic passing through $x_0$. Then there exists a uniform constant $N\in \mathbb N$ such that for each $j\in \mathbb N$ there exist at most $N$ Whitney cubes $Q\in D_j\setminus D_{j-1}$ intersecting $\gamma$. This follows from the observation that if $|x_1-x_2| \geq \delta_D(x_1)/2$, then the number of Whitney cubes $N(x_1,x_2)$ that intersect a quasihyperbolic geodesic joining $x_1, x_2$ satisfies
$$k(x_1,x_2) \simeq N(x_1,x_2).$$
See e.g.\ \cite[p.~205]{KOS}.
\end{remark}

\begin{remark}\label{quasi:remark length}
If $\gamma$ is a quasihyperbolic geodesic intersecting $Q$, then its Euclidean length inside $Q$ is bounded, up to a multiplicative constant, by $\ell(Q)$. More precisely, $\mathcal H^1(Q\cap \gamma) \lesssim \ell(Q)$.  {Here $\mathcal{H}^1(S)$ denotes the 1-dimensional Hausdorff measure of a set $S\subset \mathbb{C}$, defined by}
$$\mathcal{H}^1(S)\coloneqq \lim_{\delta \to 0} \mathcal{H}_\delta^1(S),$$
where
$$
\mathcal{H}_\delta^1(S)\coloneqq \inf \left\{ \sum_{j=1}^\infty \operatorname{diam}(U_j): S \subset \bigcup_j U_j,\, \operatorname{diam}(U_j)<\delta \right\}.
$$
\end{remark}

\subsection{Existence of geodesics}
Here, we prove that for domains $D$ satisfying the quasihyperbolic condition \eqref{quasi:condition} the  quasihyperbolic geodesics land surjectively onto the boundary $\partial D$.

\begin{lemma}\label{quasi:lemmageodesics}
Suppose that there exists a point $x_0\in D$ with
\begin{align*}
\int_D k(x,x_0)^2 \,dx<\infty.
\end{align*}
Let $z\in  \partial D$ and $z_n\in D$ be a sequence with $z_n\to z$. Also, consider quasihyperbolic geodesics $\gamma_n\colon [0,1]\to D$ from $x_0$ to $z_n$, parametrized by rescaled Euclidean arc-length. Then there exists a subsequence of $\gamma_n$ that converges uniformly to a geodesic $\gamma\subset D$, landing at $z$.

Moreover, if $z_n\in \partial D$ is a sequence with $z_n\to z\in \partial D$, and $\gamma_n$ are quasihyperbolic geodesics from $x_0$ to $z_n$ parametrized by rescaled {Euclidean} arc-length, then there exists a subsequence of $\gamma_n$ that converges uniformly to a geodesic $\gamma$ from $x_0$ to $z$.
\end{lemma}

In other words, after reparametrizing, $\gamma|_{[0,1)}\subset D$ and $\gamma(1)=z$. This lemma shows that each point $z\in \partial D$ is the landing point of a quasihyperbolic geodesic.

\begin{remark}\label{quasi:remark curves}
The proof is very similar to the discussion in \cite[pp.~273--274]{JOS}, where under the same assumptions the authors construct curves that behave like quasihyperbolic geodesics and satisfy the conclusion of the lemma. Here, we actually prove that the quasihyperbolic geodesics do in fact yield the conclusion, and that we do not need to construct new curves.
\end{remark}

\begin{proof}
The second part of the lemma can be proved exactly as the first part, so we omit that proof.

Let $z_n\in D$ be a sequence with $z_n\to z\in \partial D$, and consider the quasihyperbolic geodesics $\gamma_n\colon [0,1]\to D$, parametrized by rescaled {Euclidean} arc-length, such that $\gamma_n(0)=x_0$ and $\gamma_n(1)=z_n$. We first claim that these geodesics have uniformly bounded length. This will follow from the next lemma.

\begin{lemma}\label{quasi:lemma length}
Let $\gamma$ be a quasihyperbolic geodesic passing through $x_0$. Let $j_0\in \N$, and consider $\beta$ to be a subpath of $\gamma$ intersecting only Whitney cubes $Q\in \mathcal W(D)$ with $j(Q)\geq j_0$. Then
\begin{align*}
\length(\beta)\lesssim j_0^{-1/2}.
\end{align*}
In particular, $D$ is bounded.
\end{lemma}

\begin{proof}
By Remark \ref{quasi:remark length}, we have
\begin{align*}
\length(\beta)&\lesssim \sum_{Q: Q\cap \beta\neq \emptyset}\ell(Q) =\sum_{Q:Q\cap \beta\neq \emptyset}\ell(Q)j(Q)j(Q)^{-1}\\
&\leq \left(\sum_{Q\in \mathcal W(D)} \ell(Q)^2 j(Q)^2\right)^{1/2} \left(\sum_{Q:Q\cap \beta\neq \emptyset} j(Q)^{-2}\right)^{1/2}.
\end{align*}
The first term is comparable to $\| k(\cdot, x_0)\|_{L^2(D)}$. Indeed, since $\bigcup_{j=1}^\infty \bigcup_{Q:j(Q)=j} Q=D$, we have
\begin{align*}
\int_D k(x,x_0)^2 \,dx&= \sum_{j=1}^\infty \sum_{Q: j(Q)=j} \int_Q k(x,x_0)^2 \,dx  \\
&\simeq \sum_{j=1}^\infty  \sum_{Q: j(Q)=j} \ell(Q)^2 j(Q)^2 = \sum_{Q\in \mathcal W(D)} \ell(Q)^2j(Q)^2.
\end{align*}
Here, we also used the fact that $k(x,x_0)\simeq j(Q)$ for all $x\in Q$ with $j(Q)>1$, since the quasihyperbolic diameter of $Q$ is at most $1$. Moreover, in the case $j(Q)=1$, we also have $\int_Q k(x,x_0)^2 \,dx \simeq \ell(Q)^2 j(Q)^2$.

Hence, we have
\begin{align*}
\length(\beta)\lesssim \left(\sum_{Q:Q\cap \beta \neq \emptyset} j(Q)^{-2}\right)^{1/2}
\end{align*}
and it suffices to control the latter term. Using Remark \ref{quasi:remark number}, we may write
\begin{align*}
\sum_{Q:Q\cap \beta \neq \emptyset} j(Q)^{-2} = \sum_{j\geq j_0} \sum_{\substack{Q:Q\cap \beta\neq \emptyset \\ j(Q)=j}} j^{-2} \simeq \sum_{j\geq j_0}^\infty j^{-2} \simeq j_0^{-1}.
\end{align*}
The conclusion follows.
\end{proof}

Now, we return to the proof of Lemma \ref{quasi:lemmageodesics}. The paths $\gamma_n$ have uniformly bounded lengths by Lemma \ref{quasi:lemma length} and they stay in the set $\br D$, which is compact again by Lemma \ref{quasi:lemma length}. Applying the Arzel\`a-Ascoli theorem, we may extract a subsequence, still denoted by $\gamma_n$, that converges uniformly to a rectifiable path $\gamma \colon [0,1]\to \br D$ with $\gamma(0)=x_0$ and $\gamma(1)=z$; see \cite[Theorem 2.5.14]{BBI} for a detailed argument. We parametrize $\gamma \colon [0,1]\to \br D$ by rescaled {Euclidean} arc-length, and we now have to show that $\gamma|_{[0,1)}$ is a {quasihyperbolic} geodesic.

For this, it suffices to prove that $\gamma|_{[0,1)}$ is a path contained in $D$. It is a general fact that if a sequence of geodesics $\zeta_n$ in a \textit{length space} $(X,d)$ converges uniformly to a path $\zeta$ in $X$, then $\zeta$ is also a geodesic; see \cite[Theorem 2.5.17]{BBI}. In our case, one needs to apply this principle to all compact subpaths $\zeta\subset \subset D$ of the path $\gamma$, and suitable subpaths $\zeta_n$ of $\gamma_n$ converging to $\zeta$.

Now, we argue that $\gamma|_{[0,1)} \subset D$. Suppose for a contradiction that there exists some time $t_1\in (0,1)$ such that $\gamma(t_1)\in \partial D$, and let $t_1$ be the first such time. Note that the curve $\gamma$ cannot be constant on $(t_1,1)$, since it is parametrized by {Euclidean} arc-length. Hence, either there exists $t_2\in (t_1,1)$ such that $\gamma(t_2)\in D$, or $\{\gamma(t): t\in [t_1,1]\}$ is a non-degenerate continuum, i.e., it contains more than one point, contained in $\partial D$. The first scenario can be easily excluded, because the quasihyperbolic geodesics connecting $x_0$ to points in a small neighborhood of $\gamma(t_2)$ must remain in a fixed compact subset of $D$. Thus, the limiting path $\gamma|_{[0,t_2]}$ is also contained in the same compact set, and it cannot meet $\partial D$, a contradiction.

In the second case, suppose that there exists a point $y=\gamma(t_3)\in \partial D$, $t_3\in (t_1,1)$, with $y\neq z$. Let $y_n\in \gamma_n$ be points converging to $y$, and let $\beta_n$ be the subpath of $\gamma_n$ from $y_n$ to $z_n$.

Then, for each $j_0\in \N$, there exists $n_0\in \N$ such that for $n\geq n_0$ the path $\beta_n$ intersects only cubes $Q$ with $j(Q) \geq j_0$. Indeed, otherwise there exists a fixed cube $Q$ that is intersected by {infinitely many paths} $\beta_n$. Suppose this is the case for all $n\in \N$, by passing to a subsequence. Then
\begin{align*}
k(x_0,z_n)\leq k(x_0,Q)+1+ k(Q,z_n) \leq k(x_0,Q)+1+ k(y_n,z_n),
\end{align*}
but this is strictly less than $k(x_0,z_n)=k(x_0,y_n)+k(y_n,z_n)$ for large $n$, since $y_n\to y\in \partial D$; recall that $\gamma_n$ is a geodesic passing through $x_0,y_n$, and $z_n$. This is a contradiction.

We fix $j_0$, $n_0\in \N$, as above. Using Lemma \ref{quasi:lemma length} we conclude that
\begin{align*}
|y_n-z_n|\leq \length (\beta_n)\lesssim j_0^{-1/2}
\end{align*}
for $n\geq n_0$. Taking limits, it follows that $y=z$, a contradiction.
\end{proof}

Following \cite{JOS}, for each cube $Q\in \mathcal W(D)$ we define the \textit{shadow} $SH(Q)$ of $Q$ to be the set of points $z\in \partial D$ such that there exists a quasihyperbolic geodesic {starting at $x_0$}, passing through $Q$ and landing at $z$. We then define
\begin{align*}
s(Q)= \diam( SH(Q)).
\end{align*}
\begin{lemma}\label{quasi:shadow:compact}
For each Whitney cube $Q\in \mathcal W(D)$ the shadow $SH(Q)$ is a compact subset of $\partial D$.
\end{lemma}
\begin{proof}
Since $\partial D$ is bounded by Lemma \ref{quasi:lemma length}, it suffices to show that $SH(Q) \subset \partial D$ is closed. If $z_n$ is a sequence in $SH(Q)$ converging to $z \in \partial D$, then a sequence of geodesics $\gamma_n$ passing through $Q$ and landing at $z_n$ has a subsequence converging to a geodesic landing at $z$, by Lemma  \ref{quasi:lemmageodesics}. This limiting geodesic necessarily passes through $Q$ as well, since $Q$ is closed, so that $z \in SH(Q)$.
\end{proof}

\begin{lemma}\label{quasi:shadow sum}
We have
\begin{align*}
\sum_{Q\in \mathcal W(D)} s(Q)^2 \lesssim \int_D k(x,x_0)^2 \,dx.
\end{align*}
\end{lemma}
This was proved in \cite[p.~275]{JOS}. In the proof, the authors use the curves mentioned in Remark \ref{quasi:remark curves} instead of the quasihyperbolic geodesics, but the proof remains the same, so we omit it.

\subsection{Detours}

In this subsection, our goal is to  show that any path $\gamma\subset \C$ that intersects $\partial D$ can be modified near $\gamma\cap \partial D$ to obtain a new path $\tilde \gamma$, called a ``detour" path, that intersects $\partial D$ only at finitely many points and has certain properties. Our standing assumption, unless otherwise stated, is that $D$ satisfies the quasihyperbolic condition \eqref{quasi:condition}. We first need a technical lemma.

Two Whitney cubes $Q_1,Q_2\in \mathcal W(D)$ with, say, $\ell(Q_1)\geq \ell(Q_2)$ are \textit{adjacent} if a side of $Q_2$ is contained in a side of $Q_1$. This allows the possibility that $Q_1=Q_2$.

\begin{lemma}\label{quasi:geodesics near boundary}
For every $\varepsilon>0$ and $x\in \partial D$ there exists $r>0$ such that for all points $y\in \br B(x,r)\cap \partial D$ there exist adjacent Whitney cubes $Q_x,Q_y$ with $x\in SH(Q_x)$, $y\in SH(Q_y)$, and $\ell(Q_x)\leq \varepsilon$, $\ell(Q_y)\leq \varepsilon$. Let  $\gamma_x,\gamma_y$ be quasihyperbolic geodesics from $x_0$ to $x,y$, passing through $Q_x,Q_y$, respectively. Also, consider the subpaths $\beta_x,\beta_y$ of $\gamma_x,\gamma_y$ that intersect $Q_x,Q_y$ at only one point and land at $x,y$, respectively. Then $Q_x$ and $Q_y$ can be chosen so that we also have  $\beta_x,\beta_y\subset B(x,\varepsilon)$ and $\beta_x,\beta_y$ intersect only Whitney cubes $Q$ with $\ell(Q)\leq\varepsilon$.
\end{lemma}
\begin{proof}
We fix $\varepsilon>0$, $j_0\in \N$, and let $Q_{x}^i$, $i\in I$, be the family of cubes such that $x\in SH(Q_x^i)$ for $i\in I$ and $j(Q_x^i)=j_0$. This is a finite family, contained in $D_{j_0}$, but it also depends on $j_0$. Consider a quasihyperbolic geodesic $\gamma_x^i$ from $x_0$ to $x$ passing through $Q_{x}^i$. By Lemma \ref{quasi:lemma length}, we may choose a sufficiently large $j_0$ so that for each $i\in I$, whenever $\gamma_x^i$ is a geodesic from $x_0$ to $x$ passing through $Q_x^i$, and $\beta_x^i$ is the subpath from $Q_x^i$ to $x$, we have $\length(\beta_x^i)<\varepsilon/2$ (each Whitney cube $Q$ intersected by $\beta_x^i$ must satisfy $j(Q)\geq j_0$). In particular, each point of $\beta_x^i\subset B(x,\varepsilon)$ is very close to $\partial D$, and we may also have (by choosing an even larger $j_0$) that $\beta_x^i$ intersects only Whitney cubes $Q$ with $\ell(Q)\leq \varepsilon$. This also implies that $\ell(Q_x^i)\leq \varepsilon$.

By choosing an even larger $j_0$ we may achieve the same conclusions for all Whitney cubes $Q_y$ adjacent to $Q_x^i$, since they satisfy $j(Q_y)\geq j_0-1$. Namely, if $\gamma_y$ is a quasihyperbolic geodesic from $x_0$ to a point $y\in \partial D$ passing through a cube $Q_y$ adjacent to some $Q_x^i$, $i\in I$, then for the subpath $\beta_y$ of $\gamma_y$ from $Q_y$ to $y$ we have that $\length(\beta_y)<\varepsilon/2 $, and $\beta_y$ intersects only Whitney cubes with $\ell(Q) \leq \varepsilon$.

To finish the proof, we claim that there exists $r>0$ such that if $y\in \br B(x,r)\cap \partial D$, then there exists $i\in I$ and a  Whitney cube $Q_y$, adjacent to $Q_x^i$, such that $y\in SH(Q_y)$. Assume that this fails. Then there exists a sequence $\partial D\ni y_n\to x$ such that for all cubes $Q$ adjacent to $Q_x^i$, $i\in I$, we have $y_n\notin SH(Q)$. Consider a geodesic $\gamma_n$ from $x_0$ to $y_n$. By Lemma \ref{quasi:lemmageodesics}, after passing to a subsequence, $\gamma_n$ converges uniformly to a geodesic $\gamma$ from $x_0$ to $x$. Hence $\gamma$ intersects some cube $Q_x^i$, for some $i\in I$. By uniform convergence, for sufficiently large $n$ we must have that $\gamma_n$ intersects a cube adjacent to $Q_x^i$, a contradiction.
\end{proof}

\begin{lemma}\label{quasi:lemma detours}
Let $\gamma\colon[0,1] \to \mathbb C$ be a simple path connecting two points $a,b\in \mathbb C$. Then for each $\varepsilon>0$ we can find a detour path $\tilde \gamma$ contained in the  $\varepsilon$-neighborhood of $\gamma$, connecting $a$ and $b$, such that
\begin{enumerate}[\upshape(i)]
\item $\tilde \gamma \cap \partial D$ is a finite set,
\item the complementary components of $ D$ intersected by $\tilde \gamma$ are also intersected by $\gamma$, and
\item {if $Q \in \mathcal{W}(D)$ is a Whitney cube satisfying $Q\cap \tilde \gamma \neq \emptyset$}, then either
\begin{align*}
Q\cap \gamma&\neq \emptyset, \quad \textrm{or}\\
SH(Q)\cap \gamma&\neq \emptyset \quad \textrm{and} \quad  \ell(Q)\leq \varepsilon.
\end{align*}
\end{enumerate}
Moreover, if $a$ or $b$ lie on complementary components of $D$ and $\gamma|_{(0,1)}$ does not intersect these components, then the detour $\widetilde \gamma$ may be taken to have the same property.
\end{lemma}

In other words, $\tilde \gamma$ stays arbitrarily close to $\gamma$, connects the same points, it does not intersect any ``new" complementary components of $D$ and it does not intersect any ``new" Whitney cubes, with the exception of some ``small" Whitney cubes whose shadows intersect $\gamma$. Sets like $\partial D$ that admit such ``detours" (with similar properties) were formalized and studied by the first author in \cite{Nt}. The proof of Lemma \ref{quasi:lemma detours} relies crucially on Lemma \ref{quasi:geodesics near boundary}.

\begin{proof}
If $\gamma\cap \partial D=\emptyset$ then the statement is trivial, so we assume that $\gamma\cap \partial D\neq \emptyset$.

We fix $\varepsilon>0$ and for each $x\in \gamma \cap \partial D$ we consider a radius $r_x$ such that the conclusion of Lemma \ref{quasi:geodesics near boundary} is true for points $y\in \br B(x,r_x)\cap \partial D$. Note that $D$ is bounded, by Lemma \ref{quasi:lemma length}. {Hence $\partial D$ is compact}. We cover $\gamma \cap \partial D$ by finitely many balls $B_i\coloneqq B(x_i,r_i)$, $i=1,\dots,N$, where $r_i=r_{x_i}$.

Suppose first that $a,b\notin \partial D$. By choosing possibly smaller balls, we may assume that $a,b\notin B_i$ for all $i\in \{1,\dots,N\}$. We consider the first entry point of $\gamma$ into $\partial D$,  {as one travels along $\gamma$ from $a$ to $b$}. Assume that this point is $y_1\in \overline{B_1}\cap \partial D$, and let $z_1\in \overline{B_1}\cap \partial D$ be the last exit point of $\gamma$ from $\overline{B_1}\cap \partial D$. By Lemma \ref{quasi:geodesics near boundary} we can find paths $\gamma_{y_1},\gamma_{z_1}$ connecting $y_1,z_1$ to $x_1$, respectively, such that both paths intersect only small Whitney cubes of $D$, whose shadow intersects $\gamma$. Also the paths $\gamma_{y_1},\gamma_{z_1}$ are contained in $B(x_1,\varepsilon)$ so they are $\varepsilon$-close to $\gamma$, and they only intersect $\partial D$ at the points $y_1,x_1,z_1$. We set $\tilde \gamma$ to be the subpath of $\gamma$ from $a$ to $y_1$, concatenated with $\gamma_{y_1}$ and $\gamma_{z_1}$.

We now repeat the procedure with $\gamma$ replaced {with} its subpath from $z_1$ to $b$. Note that either $z_1\in B_1$, or $z_1\in \partial B_1$. In the first case, we necessarily have that there exists a point $z_1'\in \gamma$ ``immediately after" $z_1$ with $z_1'\notin \partial D$, so the same argument can be repeated, as in the previous paragraph, with $a$ replaced with $z_1'$.  If $z_1\in \partial B_1\cap \partial D$, then there exists a ball among $B_2,\dots,B_N$ containing $z_1$, since the balls $B_i$, $i\in \{1,\dots,N\}$ cover $\gamma \cap \partial D$. We may assume that this ball is $B_2$, so $z_1\in B_2\cap \partial D= B(x_2,r_2)\cap \partial D$. We then set $y_2=z_1$, and let $z_2\in \overline{B_2}\cap \partial D$ be the last exit point of $\gamma$ from $\overline{B_2}\cap \partial D$. Then we concatenate $\tilde \gamma$ with the paths $\gamma_{y_2}$ and $\gamma_{z_2}$ given by Lemma \ref{quasi:geodesics near boundary}, which connect $y_2$ and $z_2$ to $x_2$, respectively. One continues in this way to obtain a path from $a$ to $b$.

The cases $a\in \partial D$ or $b\in \partial D$ can be treated with a similar argument. Namely. if $a\in \partial D$, then one can use the same argument with $y_1$ replaced with $a$. The last statement in the lemma is ensured by our construction, since the detours intersect no more points of $\partial D$ than $\gamma$ does.
\end{proof}

The next statement holds in general and is independent of the quasihyperbolic condition \eqref{quasi:condition}.

\begin{lemma}\label{quasi:detours2}
Let $\gamma\colon [0,1]\to \C$ be a simple path with endpoints $a,b\in \br D$, $a\neq b$, such that $\gamma|_{(0,1)}$ does not intersect the complementary components of $D$ that possibly contain $a$ or $b$ in their boundary. Moreover, suppose that $\gamma|_{(0,1)}$ intersects finitely many complementary components $B_i$, $i\in I$, of $D$. Then there exist finitely many subpaths $\gamma_1,\dots,\gamma_m$ of $\gamma$ with the following properties:
\begin{enumerate}[\upshape(i)]
\item $\gamma_i$ is contained in $D$, except possibly for its endpoints, for each $i\in \{1,\dots,m\}$,
\item $\gamma_i$ and $\gamma_j$ intersect disjoint sets of Whitney cubes $Q\in \mathcal W(D)$ for $i\neq j$, with the exception of the endpoints, which could lie on the same cube (see \textup{(A1)} below),
\item $\gamma_1$ starts at $a_1=a$, $\gamma_m$ terminates at $b_m=b$, and in general the path $\gamma_i$ starts at $a_i$ and terminates at $b_i$ such that for each $i\in \{1,\dots,m-1\}$ we either have
\begin{enumerate}
\item[\upshape(A1)] $b_i,a_{i+1}\in \partial Q_{j_i}$ for some $Q_{j_i}\in \mathcal W(D)$, i.e., $\gamma_i$ and $\gamma_{i+1}$ have an endpoint on some Whitney cube $\partial Q_{j_i}$, or
\item[\upshape(A2)] $b_i,a_{i+1}\in \partial B_{j_i}$ for some $j_i\in I$, i.e., $\gamma_i$ and $\gamma_{i+1}$ have an endpoint on some set $\partial B_{j_i}$.
\end{enumerate}
The cubes $Q_{j_i}$ from the first alternative are distinct and the complementary components $B_{j_i}$ from the second alternative are also distinct.
\end{enumerate}
\end{lemma}

We remark that some of the subpaths $\gamma_i$ of $\gamma$ might be constant. This lemma appears, in a slightly modified version, in \cite{Nt:potential} as Lemma 2.30, where the ``peripheral disks" in the statement there, for our purpose, are replaced with ``complementary components of $D$ and Whitney cubes $Q\in \mathcal W(D)$". The proof is elementary and uses an appropriate algorithm to cut the path $\gamma$ into the desired subpaths.

If the second alternative (A2) occurs for each $\gamma_i$ in Lemma \ref{quasi:detours2}, then we will call the paths $\gamma_i$ and the complementary components $B_{j_i}$ a transboundary chain. More precisely:

\begin{definition}
\label{chain}
Let $D \subset \mathbb{C}$ be a domain. Let $\mathcal{C}\coloneqq(\gamma_1, B_1, \dots, \gamma_{m-1}, B_{m-1}, \gamma_{m})$, $m \geq 1$, be a collection of paths $\gamma_i$ in $\br D$ and complementary components $B_i$ of $D$; if $m=1$ then $\mathcal C\coloneqq (\gamma_1)$. The collection $\mathcal C$ is called a \textit{transboundary chain} of $D$ if
\begin{enumerate}
\item each path $\gamma_i$ is contained in $D$, except possibly for its endpoints $a_i,b_i$, and $a_i\neq b_i$,
\item each $B_i$ is a complementary component of $D$, and the components $B_i$ are all distinct, and
\item for each $i\in \{1,\dots,m-1\}$ we have $b_i, a_{i+1} \in \partial B_i$.
\end{enumerate}
The points $a_1$,$b_{m}$ are called the \textit{endpoints} of the transboundary chain $\mathcal C$ and they might lie in $D$, but we also allow the possibility that $a_1$ or $b_{m}$ belong to $\partial B_0$ or $\partial B_{m}$, where $B_0$ or $B_{m}$ are complementary components of $D$ such that $B_0, B_1, \dots, B_{m-1}$ or $B_1,\dots,B_m$ are all distinct, respectively. In this case, we can add $B_0$ or $B_{m}$ to  the chain in the obvious way and obtain a transboundary chain of the form $(B_0,\gamma_1, \dots,\gamma_{m})$ or $(\gamma_1, \dots,\gamma_{m},B_m)$, respectively.
\end{definition}

With this definition, if we combine the preceding two lemmas we have:

\begin{corollary}\label{quasi:detours:cor}
Let $\gamma\colon [0,1]\to \C$ be a simple path with endpoints $a,b\in \br D$, $a\neq b$, such that $\gamma|_{(0,1)}$ does not intersect the complementary components of $D$ that possibly contain $a$ or $b$ in their boundary. Then for each $\varepsilon>0$ there exist paths $\gamma_1,\dots, \gamma_{m}$ and (if $m\geq 2$) complementary components $B_1, \dots, B_{m-1}$ of $D$ with the following properties:
\begin{enumerate}[\upshape(i)]
\item $\gamma_1$ starts at $a$ and $\gamma_m$ terminates at $b$,
\item each $B_i$ is intersected by $\gamma |_{(0,1)}$,
\item $(\gamma_1, B_1, \dots, \gamma_{m-1}, B_{m-1}, \gamma_{m})$ is a transboundary chain of $D$,
\item $\gamma_i$ and $\gamma_j$ intersect disjoint sets of Whitney cubes $Q\in \mathcal W(D)$ for $i\neq j$,
\item if $Q \in \mathcal{W}(D)$ is a Whitney cube with $Q\cap\gamma_i \neq \emptyset$, then either
\begin{align*}
Q\cap \gamma&\neq \emptyset, \quad \textrm{or}\\
SH(Q)\cap \gamma&\neq \emptyset \quad \textrm{and} \quad  \ell(Q)\leq \varepsilon,
\end{align*}
\item if $Q \in \mathcal{W}(D)$ is a Whitney cube with $Q\cap\gamma_i \neq \emptyset$, then $Q\cap \gamma_i$ is contained in the union of two line segments.
\end{enumerate}

\end{corollary}

The important properties of the paths $\gamma_i$ are that they intersect disjoint sets of Whitney cubes, all of which are intersected by $\gamma$, with the exception of some ``small" cubes whose shadow intersects $\gamma$.

\begin{remark}\label{quasi:detours:cor:rem}
If $a\in \partial D$, then we may consider the component $B_0$ of $\C\setminus D$ that contains $a$ and obtain a transboundary chain of the form $(B_0,\gamma_1,\dots,\gamma_m)$ in the above corollary. Similarly, if $b\in B_m$, we obtain a chain of the form $(\gamma_1,\dots,\gamma_m,B_m)$. Note that $B_0$ and $B_m$ are necessarily distinct from $B_1,\dots,B_{m-1}$ because $\gamma|_{(0,1)}$ does not intersect $B_0$ or $B_m$, by assumption.

\end{remark}

\begin{proof}
We first apply Lemma \ref{quasi:lemma detours} to obtain a detour path $\widetilde \gamma$. Then we apply Lemma \ref{quasi:detours2} to $\widetilde \gamma$ to obtain paths $\widetilde\gamma_1,\dots,\widetilde \gamma_M$. If the alternative (A1) in Lemma \ref{quasi:detours2} occurs for two paths $\widetilde \gamma_i,\widetilde \gamma_{i+1}$, then we connect their endpoints that lie on the boundary of a common Whitney cube with a line segment in that cube. In this way we obtain a new family of paths. We repeat this procedure, until no two paths have endpoints on the same cube. This gives us a collection of paths $\gamma_1,\dots, \gamma_{m}$ that intersect disjoint sets of Whitney cubes, as required in (iv); it is important here that the cubes in which we add the line segments are distinct, as provided by Lemma \ref{quasi:detours2}(iii).

We note that (v) holds by Lemma \ref{quasi:lemma detours}(iii), since the paths $\widetilde \gamma_i$ are subpaths of $\widetilde \gamma$, and we only perform concatenations inside cubes that are also intersected by $\widetilde \gamma$. Moreover, (i), (ii), and (iii) follow from Lemma \ref{quasi:detours2}(i) and (iii), since we have eliminated the first alternative (A1). We remark that in order to have a transboundary chain, it is required in the definition that if the endpoints $a$ of $\gamma_1$ or $b$ of $\gamma_m$ lie in complementary components of $D$, then each of these components has to be distinct from $B_1,\dots,B_{m-1}$; this is guaranteed by our assumptions as noted in Remark \ref{quasi:detours:cor:rem}.

It remains to modify the paths $\gamma_i$ so that they satisfy (vi). Since the paths $\gamma_i$ intersect disjoint sets of Whitney cubes, we can do the modifications individually in each $\gamma_i$. Suppose that $\gamma_i$ is a path both of whose endpoints $a_i,b_i$ lie on complementary components of $D$, but otherwise $\gamma_i$ is contained in $D$; since $\gamma_i$ is part of a transboundary chain, we must have $a_i\neq b_i$.

Moreover, suppose that $\gamma_i\colon [0,1]\to \C$ is parametrized so that it runs from $a_i$ to $b_i$; that is $\gamma_i(0)=a_i$ and $\gamma_i(1)=b_i$. Let $Q_0$ be any Whitney cube intersected by $\gamma_i$. There exists a point $\gamma_i(t_0)=z_0\in \partial Q_0$ that is the last exit point of $\gamma_i$ from  $Q_0$, i.e., the path $\gamma_i|_{(t_0,1)}$ does not intersect the cube $Q_0$. Hence, there exists $t>t_0$ arbitrarily close to $t_0$ so that $\gamma_i(t)$ intersects a cube $Q_1$, adjacent to $Q_0$. We let $\gamma(t_1)=z_1$, $t_1>t_0$, be the last exit point  of $\gamma_i$ from $Q_1$. Continuing in this way, we obtain a chain of distinct adjacent cubes $Q_0,Q_1,\dots$ such that $Q_n \to b_i$ as $n\to \infty$. In the same way we obtain a chain of distinct adjacent cubes $Q_0,Q_{-1},\dots$ such that $Q_{n}\to a_i$ as $n\to-\infty$. We note that a cube $Q\in \mathcal W(D)$ might appear twice in the sequence $Q_n$, $n\in \mathbb Z$, but this can only be the case for finitely many cubes, because $Q_n\to a_i$ as $n\to -\infty$ and $Q_n\to b_i$ as $n\to\infty$. We truncate the chain $Q_n$ to obtain a new chain $\widetilde Q_n$, $n\in \Z$, that has the above properties, but no cube appears twice. Then we connect the centers of adjacent cubes $\widetilde Q_n$ with line segments and this gives us a path ${\Gamma_i}$ connecting $a_i$ to $b_i$. The path $\Gamma_i$ is the desired path that satisfies (vi), and also all the other required properties of the corollary, since it only intersects Whitney cubes that are also intersected by the original path $\gamma_i$.

If $i=1$, then $\gamma_1$ starts at $a_1=a$, which might lie in the interior of $D$. One can easily modify the above argument by connecting $a$ to the center of the Whitney cube that possibly contains it and obtain the path $\Gamma_1$ as in (vi).  The same holds for $\gamma_{m}$, which might terminate in a point in $D$.
\end{proof}

\subsection{Absolute continuity lemmas}\label{ac:section}
This subsection contains some absolute continuity results that will be crucial for the proof of the main theorem.

Let $D\subset \C$ be an open set. A  function $f\colon D \to \C$ lies in the Sobolev space $W^{1,2}(D)$ if $f\in L^2(D)$ and $f$ has weak derivatives of first order that lie in $L^2(D)$. In particular, if $f\colon D\to f(D)\subset \C$ is a conformal map with bounded domain and range, then $f\in W^{1,2}(D)$, since
\begin{align*}
\int_D |f'|^2 = \area(f(D))<\infty.
\end{align*}

\begin{proposition}\label{ac:lemma Sobolev}
Suppose that {the domain} $D\subset \C$ satisfies \eqref{quasi:condition}, and let $f\colon D\to \C$ be a {continuous} function in $W^{1,2}(D)$ that extends continuously to $\br D$. Also, let $x\in \C$ be arbitrary, and denote by $\gamma_r(t)=x+re^{it}$, $t\in [0,2\pi]$, the circular path around $x$ at distance $r$. Then for a.e.\ $r\in (0,\infty)$ we have
\begin{align*}
\mathcal H^1( f(\gamma_r\cap \partial D))=0.
\end{align*}
\end{proposition}
This proposition is a variant of \cite[Proposition 5.3]{Nt}, but the proof is almost identical, and is based on the detours given by Lemma \ref{quasi:lemma detours}. In fact, the statement and proof date back to the original work of Jones and Smirnov in \cite[Proposition 1]{JOS}.

\begin{corollary}\label{ac:zero measure}
Suppose that $D\subset \C$ satisfies \eqref{quasi:condition}. Then $\area(\partial D)=0$.
\end{corollary}

\begin{proof}
It suffices to apply Proposition \ref{ac:lemma Sobolev} to the identity function and integrate over all circles, using Fubini's theorem.
\end{proof}

\begin{lemma}\label{ac:lemma Real}
Let $Z\subset \R$ be a closed set and $f\colon Z \to \C$ be a continuous function. Consider the linear extension of $f$ in each bounded complementary open interval of $Z$ and extend $f$ by a constant in the unbounded complementary intervals of $Z$, if any. This yields a continuous extension $f\colon \R\to \C$. Suppose that $K\subset Z$ is a closed set containing $\partial Z$. If $f$ is locally absolutely continuous on each component of $Z\setminus K=\inter(Z)\setminus K$ and $\mathcal H^1(f(K))=0$, then $f'=0$ a.e.\ on $K$ and for all $x,y\in \R$ with $x\leq y$ we have
\begin{align*}
|f(x)-f(y)|\leq \int_{[x,y]}|f'|,
\end{align*}
where the latter {integral} might be infinite. In particular, if $(x_i,y_i)$, $i\in \N$, are the bounded components of $\R\setminus Z$, then
\begin{align*}
|f(x)-f(y)|\leq \int_{[x,y]\cap  (Z\setminus K)}|f'| + \sum_{[x,y]\cap(x_i,y_i)\neq \emptyset} |f(x_i)-f(y_i)|,
\end{align*}
for all $x,y\in \R$.
\end{lemma}

Here, a $\C$-valued function is absolutely continuous if its real and imaginary part are. The proof of this lemma is elementary and can be derived from the Banach-Zaretsky theorem \cite[Theorem 4.6.2, p.\ 196]{BC}; see also \cite[Lemma 6.4]{Nt}.

\subsection{Distortion estimates}\label{distortion:section}

We end this section with distortion estimates on Whitney cubes that will be used in subsequent sections. In the following, $D,D^{*} \subsetneq \C$ are domains and $f\colon D \to D^{*}$ is a conformal map.

\begin{lemma}[Koebe's distortion theorem]\label{quasi:koebe distortion}
Let $z_0\in D$, {$0<r\leq \dist(z_0,\partial D)$}, and $0<c<1$. Then we have
\begin{equation}
\label{distortion1}
|f'(x)| \simeq \frac{|f(y)-f(z)|}{|y-z|}
\end{equation}
for all $x,y,z\in B(z_0,cr)$ with constants depending only on $c$. In particular, $f$ is bi-Lipschitz in $B(z_0,cr)$.
\end{lemma}
See for instance \cite[Chapter 1.3, pp.~8--9]{Po} or \cite[Theorem 2.9]{MCM}.

\begin{lemma}
\label{quasi:koebe}
Let $Q \in \mathcal{W}(D)$ be a Whitney cube and let $A \subset Q$ be a dyadic cube of deeper level. Then we have
$$\diam(f(Q)) \simeq \dist(f(Q),\partial D^*)$$
and
$$\mint{-}_{A} |f'| \simeq \mint{-}_{Q}|f'|,$$
with constants independent of $f, Q, A$.

\end{lemma}

\begin{proof}
The second part follows directly from \eqref{distortion1}, fixing $y \in A$ and noting that
\begin{align*}
|f'(x)| \simeq |f'(y)|
\end{align*}
for all $x \in Q$. For the first part, let $r\coloneqq\dist(Q,\partial D)$ and let $z_0$ be the center of $Q$. Note that by condition (2) of the Whitney decomposition, there is a uniform constant $0<c<1$ such that $Q \subset B(z_0,cr)\subset B(z_0,r)\subset D$ and we have $r\simeq \ell(Q)$. By the version of Koebe's distortion theorem in \cite[Corollary 1.4]{Po}, we have
\begin{align}
\label{distortion2}
r|f'(x)| \simeq \dist(f(x),\partial f(B(x,r))) \leq \dist(f(x),\partial D^*)
\end{align}
for $x \in Q$. In fact, the reverse inequality is also true, as one can see by applying \cite[Corollary 1.4]{Po} to $f^{-1}$.
Therefore, using \eqref{distortion1}, for $y,z\in Q$ we have
\begin{align*}
\dist(f(x),\partial D^*) \simeq r |f'(x)| \simeq  \ell(Q) \frac{|f(y)-f(z)|}{|y-z|}.
\end{align*}
Since $x,y,z\in Q$ are arbitrary, it follows that
\begin{align*}
\dist(f(Q),\partial D^*) \simeq \diam(f(Q)). & \qedhere
\end{align*}
\end{proof}

\section{Topological preliminaries}
\label{sec3}
In this section, we collect some general topological facts that will be needed later. Most of these facts might be considered quite standard, but we give the proofs for the sake of completeness. We first enumerate some well-known facts from planar topology that will be used repeatedly.

\begin{enumerate}[(PT1)]

\item (Zoretti's theorem, \cite[Corollary 3.11, p.~35]{WHY}) Let $K$ be a component of a compact set $M$ in the plane, and let $\varepsilon>0$. Then there exists a Jordan curve $\gamma$ that encloses $K$, does not intersect $M$, and is contained in the $\varepsilon$-neighborhood of $K$.

\item (\cite[Chapter V, Theorem 10.2, p.~106]{Newman}) If $\gamma$ is a Jordan curve contained in a domain $U \subset \RiemannSphere$, then $U \setminus \gamma$ has precisely two connected components, whose boundaries are $\gamma \cup F_1$ and $\gamma \cup F_2$. Here $F_1$ and $F_2$ are the unions of the components of $\partial U$ in each of the two complementary components of $\gamma$.

\item (Direct consequence of \cite[Chapter V, Theorem 8.2, p.~101]{Newman} and \cite[Chapter V, Theorem 14.2, p.~117]{Newman}) If $A$ is a totally disconnected closed subset of a domain $U \subset \RiemannSphere$, then $U\setminus A$ is connected. We note that this also holds if $A$ is only assumed to be \textit{relatively} closed in $U$.

\end{enumerate}

Now, for the remainder of this section, the letters $\Omega$ and $\Omega^{*}$ will denote domains in $\RiemannSphere$ each containing the point $\infty$, and $f\colon  \Omega \to \Omega^{*}$ a homeomorphism with $f(\infty)=\infty$.

\subsection{Boundary correspondence}

We first prove that $f$ induces a correspondence between the components of $\partial \Omega$ and the components of $\partial \Omega^*$, in the following sense.

\begin{proposition}
\label{proposition:topological}
For each component $b$ of $\partial \Omega$, define $f^{*}(b)$ as the component $b^*$ of $\partial \Omega^{*}$ such that $\{f(z_n)\}_{n\in \N}$ accumulates at $b^{*}$ whenever $\{z_n\}_{n\in \N}$ is a sequence in $\Omega$ accumulating at $b$. Then $f^*$ is well-defined and maps the set of boundary components of $\Omega$ bijectively onto the set of boundary components of $\Omega^{*}$.

\end{proposition}

See also \cite[Section 1]{SCH1}.

\begin{proof}
First, note that if $z_n$ is a sequence in $\Omega$ accumulating at a boundary component $b$, then $f(z_n)$ must accumulate on $\partial \Omega^*$, since $f\colon \Omega \to \Omega^*$ is a homeomorphism. To prove that $f^*$ is well-defined, suppose for a contradiction that there are two sequences $z_n$ and $z_n'$ in $\Omega$ accumulating at $b$, but $f(z_n)$ and $f(z_n')$ accumulate at two distinct boundary components of $\Omega^*$, say $b^*$ and $b'^{*}$. By (PT1), we can find a Jordan curve $\gamma^*$ in $\Omega^*$ that separates $b^*$ from $b'^{*}$. Let $\gamma=f^{-1}(\gamma^*) \subset \Omega$. By (PT2), each of $\Omega^*\setminus \gamma^*$ and  $\Omega\setminus \gamma$ has precisely two components and since $f$ is a homeomorphism, it follows that each of the components of $\Omega\setminus \gamma$ is mapped bijectively to a component of $\Omega^*\setminus \gamma^*$. Hence, $z_n$ and $z_n'$  eventually lie in different components of $\Omega\setminus \gamma$, which implies that $b$ intersects both components of $\RiemannSphere \setminus\gamma$, contradicting the fact that $b$ is a connected subset of $\partial \Omega \subset \RiemannSphere \setminus \gamma$. This proves that $f^*$ is well-defined, and the same argument with $f$ replaced with $f^{-1}$ shows that if $b$ and $b'$ are two distinct boundary components of $\Omega$, then $f^*(b) \neq f^*(b')$, i.e., $f^{*}$ is injective. Finally, if $b^*$ is a boundary component of $\Omega^*$ and $w$ is any point of $b^*$, then there is a sequence $z_n$ in $\Omega$ with $f(z_n) \to w$. This sequence $z_n$ necessarily accumulates at a boundary component $b$ of $\Omega$, and we get $f^*(b)=b^*$. This shows that $f^*$ is surjective, completing the proof of the proposition.
\end{proof}

We will denote by $B$ the closed component of $\widehat{\C}\setminus \Omega$ that is bounded by $b$, and similarly $B^*$ is bounded by $b^*$. The map $f^*$ extends to the set of complementary components of $\Omega$, and if $b$ is mapped to $b^*$, then $B$ is mapped to $B^*$. With this notation, the following lemma is a direct consequence of (PT2):

\begin{lemma}\label{lemma:surround}
Let $\gamma$ be a Jordan curve in $\mathbb{C}$, and suppose that $\gamma \subset \Omega$. If $U$ denotes the bounded component of $\C\setminus \gamma$ and $U^*$ denotes the bounded component of $\C\setminus f(\gamma)$, then $f(U\cap \Omega)=U^*\cap \Omega^*$ and a complementary component $B$ of $\Omega$ is contained in $U$ if and only if $B^*$ is contained in $U^*$.
\end{lemma}

We note that one can also define the boundary correspondence map $h^*$ for any domains $D,D' \subset \RiemannSphere$ and any homeomorphism $h\colon D \to D'$. With this remark, we have the following consequence of Proposition \ref{proposition:topological} and (PT3).

\begin{lemma}\label{lemma:extension}
Let $U,U'$ be domains in $\widehat{\C}$ and let $A,A'$ be totally disconnected relatively closed subsets of $U,U'$, respectively. Let $h$ be a homeomorphism of $U\setminus A$ onto $U'\setminus A'$ such that $h^*(A)=A'$. Then $h$ has a unique extension to a homeomorphism of $U$ onto $U'$.
\end{lemma}

We conclude this subsection with the following connectedness lemma.

\begin{lemma}
\label{lemma:connected}
Let $\mathcal{C}$ be a transboundary chain of $\Omega$, as in Definition \ref{chain}. Then the set
$$\mathcal{K} \coloneqq \bigcup_{\gamma_i\in \mathcal C} \br{f(\gamma_i)} \cup \bigcup_{B_i\in \mathcal C} B_i^*$$
is a continuum and therefore
$$\diam(\mathcal K) \leq \sum_{\gamma_i\in \mathcal C} \diam(f(\gamma_i)) +\sum_{B_i\in \mathcal C} \diam(B_i^*).$$
\end{lemma}
The paths $\gamma_i\colon [0,1]\to \C$ do not have endpoints necessarily in $\Omega$, but $\gamma_i|_{(0,1)}\subset \Omega$. In view of this, $f(\gamma_i)$ is understood as $f(\gamma_i |_{(0,1)})$.

\begin{proof}
We suppose that $\mathcal C = (\gamma_1,B_1,\dots,\gamma_{m-1},B_{m-1},\gamma_m)$. The argument is the same if $\mathcal C$ is of the form $(B_0,\gamma_1,\dots,\gamma_m)$ or $(\gamma_1,\dots,\gamma_m,B_m)$.

First, note that if $z_n$ is a sequence in $\gamma_1 |_{(0,1)}$ that converges to the endpoint $b_1$ of $\gamma_1$ that lies in $\partial B_1$, then $f(z_n)$ must accumulate at $B_1^*$, by the definition of $f^*$. Hence $\overline{f(\gamma_1)} \cap B_1^* \neq \emptyset$. Now, note that $\overline{f(\gamma_1)}$ is a continuum, since $f\colon \Omega \to \Omega^*$ is continuous. It follows that $\overline{f(\gamma_1)} \cup B_1^*$ is the union of two continua with non-empty intersection, and thus is a continuum as well. The same argument with $\gamma_1$ replaced with $\gamma_2$ shows that $B_1^* \cup \overline{f(\gamma_2)}$ is a continuum, and so is $\overline{f(\gamma_1)} \cup B_1^* \cup \overline{f(\gamma_2)}$. Repeating this argument, we get that $\mathcal K$ is a continuum, as required.

The last inequality in the statement of the lemma follows from the general fact that if $C$ is a connected metric space and $C=A\cup B$, then $\diam(C)\leq \diam(A)+\diam(B)$.
\end{proof}

\subsection{Cluster sets}

Equivalently, the map $f^*$ of Proposition \ref{proposition:topological} can be defined using the notion of a cluster set. If $E \subset \partial \Omega$ is closed, then the \textit{cluster set of $f$ at $E$ (with respect to $\Omega$)} is defined by
$$
\operatorname{Clu}(f;E) \coloneqq \bigcap_{\varepsilon>0} \overline{f(N_{\varepsilon}(E)\cap \Omega)},
$$
where $N_\varepsilon(E)$ denotes the open $\varepsilon$-neighborhood of the set $E$. In particular, the cluster set of $f$ at a boundary point $z_0 \in \partial \Omega$ is
$$
\operatorname{Clu}(f;z_0)= \bigcap_{\varepsilon>0} \overline{f(B(z_0,\varepsilon)\cap \Omega)}.
$$
Note that $\operatorname{Clu}(f;E) \subset \partial \Omega^*$, since $f\colon \Omega \to \Omega^*$ is a homeomorphism. Moreover, the cluster set $\operatorname{Clu}(f;E)$ is the intersection of a decreasing sequence of compact sets, and hence is compact as well. It is immediate to see that $\operatorname{Clu}(f;E)$ is precisely the set of accumulation points of $\{f(z_n)\}_{n\in \N}$ whenever $\{z_n\}_{n\in \N}$ is a sequence in $\Omega$ converging to a point of $E$. In particular, if $b$ is a boundary component of $\Omega$, then $f^*(b)=\operatorname{Clu}(f;b)$.

We will also need the following proposition, which essentially asserts that $f$ ``maps" boundary points to boundary continua.

\begin{proposition}
\label{proposition:topological2}
Suppose that each boundary component of $\Omega$ is a Jordan curve or a single point and that for each $\varepsilon>0$ there are at most finitely many Jordan curves in $\partial \Omega$ with diameter greater than $\varepsilon$. Then for every $z_0 \in \partial \Omega$, the cluster set $\operatorname{Clu}(f;z_0)$ is a continuum (possibly degenerate).
\end{proposition}

\begin{proof}
First note that if $\{z_0\}$ is a boundary component of $\Omega$, then $\operatorname{Clu}(f;z_0)=f^*(\{z_0\})$, which is a component of $\partial \Omega^*$ by the definition of $f^*$, so it is a continuum. We can therefore assume that the boundary component of $\Omega$ containing $z_0$ is a Jordan curve. Now, suppose that this curve is the unit circle, and that all the other boundary components of $\Omega$ are single points. Then, in this case, for each $\varepsilon>0$, the set $B(z_0,\varepsilon)\cap \Omega$ is connected by (PT3). It follows that the cluster set
$$
\operatorname{Clu}(f;z_0)= \bigcap_{\varepsilon>0} \overline{f(B(z_0,\varepsilon)\cap \Omega)}
$$
is a decreasing intersection of continua, and thus has to be a continuum as well. This proves the result in this special case.

For the general case, let $b$ be the Jordan curve in $\partial \Omega$ containing $z_0$. Using the Schoenflies theorem  \cite[Corollary 2.9, p.~25]{Po}, we may map $b$ to the unit circle with a global homeomorphism of $\widehat{\C}$ that fixes $\infty$. We can therefore assume without loss of generality that $z_0$ lies in the unit circle $b\subset \partial \Omega$. Now, if $\partial \Omega$ contains only finitely many Jordan curves, then they have positive distance from $b$, so the argument in the previous paragraph can be used to obtain the conclusion that $\operatorname{Clu}(f;z_0)$ is a continuum. We thus assume that there are infinitely (and thus countably) many Jordan curves in $\partial \Omega$, different from $b$, each bounding a closed Jordan region $B_i$, $i\in \N$. Now, we are going to use Moore's decomposition theorem \cite{MO} with the following formulation found in  {\cite[Corollary 6A, p.~56]{Daverman:decompositions}}:

\begin{theorem}\label{Decomposition}
Let $\{B_i\}_{i\in \N}$ be a sequence of closed Jordan regions in $\widehat {\C}$ with diameters converging to $0$, and let $U$ be an open set containing $\bigcup_{i\in \N} B_i$. Then there exists a continuous surjective map $h\colon \widehat \C\to \widehat \C$, which is the identity outside $U$, such that $h$ induces a decomposition of $\widehat \C$ into the sets $\{B_i\}_{i\in \N}$ and points. Specifically,  there exist countably many points $p_i\in U$, $i\in \N$, such that $h^{-1}(p_i)=B_i$ for each $i\in \N$, and the map $h\colon  \widehat\C\setminus \bigcup_{i\in \N} B_i \to \widehat\C\setminus \{p_i\colon i\in \N\}$ is bijective.
\end{theorem}

Note that our assumption on $\partial \Omega$ implies that $\diam(B_i)\to 0$ as $i\to\infty$. Moore's theorem then provides us with a map $h$ that is the identity in the closed unit disk and maps each $B_i$ to a point $p_i$. We may assume that $h(\infty)=\infty$. Moreover, $h$ is a homeomorphism when restricted to $\Omega$ (by the invariance-of-domain theorem) and maps the unit circle $b$ identically to itself. The boundary components of $h(\Omega)$ are precisely the images under $h$ of the boundary components of $\Omega$ by Proposition \ref{proposition:topological}. Hence, all boundary components of $h(\Omega)$ other than $b$ are single points. Now, Proposition \ref{proposition:topological2} will follow from the first part of the proof and the next lemma.

\begin{lemma}
\label{lemma:clusterequality}
We have
$$\operatorname{Clu}(f, z_0) = \operatorname{Clu}(f \circ h^{-1},z_0),$$
where the cluster set on the left is with respect to $\Omega$ and the cluster set on the right with respect to $h(\Omega)$.
\end{lemma}

\begin{proof}
Let $w \in \operatorname{Clu}(f, z_0)$. Then there is a sequence $z_n$ in $\Omega$ with $z_n \to z_0$ and $f(z_n) \to w$. By the continuity of $h$ on $\RiemannSphere$, we get that $h(z_n)$ is a sequence in $h(\Omega)$ with $h(z_n) \to h(z_0)=z_0$, and $f \circ h^{-1} (h(z_n)) \to w$. It follows that $w \in \operatorname{Clu}(f \circ h^{-1},z_0)$, which proves the direct inclusion.

For the reverse inclusion, let $w \in \operatorname{Clu}(f \circ h^{-1},z_0)$. Then there is a sequence $z_n$ in $\Omega$ with $h(z_n) \to z_0$ and $f \circ h^{-1} (h(z_n)) \to w$, i.e., $f(z_n) \to w$. If $z_n$ does not converge to $z_0$, then there is a subsequence $z_{n_j}$ that converges to some $z_0'$ with $z_0' \neq z_0$. Note that by Proposition \ref{proposition:topological} applied to $h^{-1}$, the point $z_0'$ necessarily belongs to the unit circle $b$. But $h(z_{n_j}) \to h(z_0')$, again by continuity, from which it follows that $h(z_0') = z_0$. Since $h$ is the identity on $b$, this gives $z_0'=z_0$, a contradiction. Thus, $z_n \to z_0$, and we get that $w \in \operatorname{Clu}(f, z_0)$, as required. This completes the proof of the lemma.
\end{proof}

Note that the cluster set $\operatorname{Clu}(f \circ h^{-1},z_0)$ is connected, by the first part of the proof applied to the homeomorphism $f \circ h^{-1}$ on $h(\Omega)$. Proposition \ref{proposition:topological2} then follows directly from Lemma \ref{lemma:clusterequality}.
\end{proof}

\begin{remark}
The conclusion of Proposition \ref{proposition:topological2} is not necessarily true if the boundary components of $\Omega$ are not assumed to be Jordan curves or points. To see this, consider a conformal map from the complement of a figure eight to the complement of a disk.
\end{remark}

If the diameters of the complementary components of $\Omega$ and $\Omega^*$ converge to $0$, then the map $f^*$ is continuous in a sense:

\begin{lemma}\label{lemma:convergenceCluster}
Suppose that for each $\varepsilon>0$ there are at most finitely many  complementary components of $\Omega$ and $\Omega^*$ with diameter greater than $\varepsilon$. Let $B$ be a component of $\widehat \C\setminus \Omega$ and $z_0\in \partial B$. If $z_n\in \widehat{\C} \setminus B$ is a sequence converging to $z_0$, then there exists a subsequence of $z_n$, still denoted by $z_n$, such that either
\begin{enumerate}[\upshape(i)]
\item $z_n\in \Omega$ for all $n\in \N$ and $f(z_n)$ converges to a point of $\clu(f,z_0)$, or
\item $z_n\in B_n$, where $B_n$ is a component of $\widehat{\C}\setminus \Omega$ for each $n\in \N$, and $B_n^*=f^*(B_n)$ converges to a point of $\clu(f,z_0)$ in the Hausdorff sense.
\end{enumerate}
In particular, if $B_n\subset \widehat{\C}\setminus \Omega$ converges to $z_0$, then $B_n^*$ is contained in arbitrarily small neighborhoods of $\clu(f,z_0)\subset \partial B^*$ as $n\to\infty$.
\end{lemma}
\begin{proof}
If there are infinitely many $n\in \N$ with $z_n\in \Omega$, we may assume that this is the case for all $n\in \N$ and then the first alternative occurs by the definition of the cluster set. If there are infinitely many $n\in \N$ with $z_n\notin \Omega$, then after passing to a subsequence we may find a sequence $B_n$ of distinct complementary components of $\Omega$ such that $z_n\in B_n$ for all $n\in \N$. Since $\diam(B_n)\to 0$, we can find points $z_n'\in \Omega$ with $|z_n-z_n'|\to 0$, and thus $z_n'\to z_0$. Moreover, if $z_n'$ is sufficiently close to $B_n$, then we may also have that $\dist(f(z_n'), B_n^*) \to 0$ by Proposition \ref{proposition:topological}. Since $\dist(f(z_n'),\clu(f,z_0)) \to 0$, we obtain $\dist(B_n^*,\clu(f,z_0))\to 0$. Since the diameters of $B_n^*$ converge to $0$, it follows that $B_n^*$ converges to a point of $\clu(f,z_0)$, after passing to a further subsequence.
\end{proof}

The following technical lemma asserts that if a cluster set $\clu(f,z_0)$ is ``big" then all ``crosscuts" shrinking to $z_0$ are mapped to sets having ``big" diameter. This lemma will be crucially used in Section \ref{sec5} in order to prove that a conformal map from a circle domain satisfying the quasihyperbolic condition onto another circle domain cannot blow up a boundary point to a boundary circle.

\begin{lemma}\label{lemma:crosscut}
Suppose that each boundary component of $\Omega$ and $\Omega^*$ is a Jordan curve or a single point and that for each $\varepsilon>0$ there are at most finitely many Jordan curves in $\partial \Omega$ and $\partial \Omega^*$ with diameter greater than $\varepsilon$.

Let $B$ be a complementary component of $\Omega$ and consider points $z_0\in \partial B$ and $w_0\in \clu(f,z_0)\subset B^*$. Suppose that $\gamma \subset \C \setminus \{z_0\}$ is a closed curve winding once around $z_0$.

Then for each $\eta>0$ there exists $\delta>0$ such that if $\gamma\subset B(z_0,\delta)$, then there exists a point $z\in \gamma\cap \Omega$  with $|f(z)-w_0|<\eta$.

In particular, for each $\eta>0$ there exists $\delta>0$ such that if $ \gamma \subset B(z_0,\delta)$, then there exist points $z_1,z_2\in \gamma\cap \Omega$ with $$|f(z_1)-f(z_2)| \geq \diam(\clu(f,z_0))-\eta.$$
\end{lemma}

\begin{proof}
Suppose first that $\partial \Omega$ has finitely many components. By Proposition \ref{proposition:topological}, $\partial \Omega^*$ also has finitely many components. Note that $B$ is either a single point $\{z_0\}$ or it is a closed Jordan region. In the latter case, using the Schoenflies theorem we may assume that $B$ is the closed unit disk $\br \D$. Again, using the Schoenflies theorem, if necessary, we assume that $B^*=\{w_0\}$ or $B^*= \br\D$. Suppose that $\delta$ is so small that $B(z_0,\delta)$ does not intersect any complementary components of $\Omega$, except for $B$.

Since $w_0\in \clu(f,z_0)$, there exists a sequence $z_n \in \Omega$ with $z_n\to z_0$ such that $w_n\coloneqq f(z_n) \to w_0$. Note that for sufficiently large $n$ the points $w_n$ and $w_{n+1}$ lie in an annulus of the form $1<|w|<R$ if $B^*=\br \D$ or $0<|w-w_0|<R$ if $B^*=\{w_0\}$, contained in $\Omega^*$; recall that $\Omega^*$ has finitely many complementary components. Therefore, $w_n$ and $w_{n+1}$ can be connected with a path that lies in $\Omega^*$ and has length at most $2|w_n-w_{n+1}|$. By connecting $w_n$ to $w_{n+1}$ with paths in $\Omega^*$ we may obtain a path $\Gamma^*\colon [0,1) \to \Omega^*$ such that $\Gamma^*(0)=w_1$,  $\lim_{t\to 1} \Gamma^*(t) =w_0$, and $\Gamma^*(t_n)=w_n$, where $t_n\in [0,1)$ is a sequence converging to  $1$. Observe that the set $\Gamma\coloneqq f^{-1}(\Gamma^*)\subset \Omega$ intersects arbitrarily small neighborhoods of $z_0$, since it contains $z_n$ for all $n\in \N$.

We let $\eta>0$ be arbitrary, and fix a large $N\in \N$ such that $\Gamma^*|_{[t_N,1)} \subset B(w_0,\eta)$. Let $\delta>0$ be so small that $f^{-1}(\Gamma^*|_{[0,t_N]})\cap \br B(z_0,\delta)=\emptyset$. We now consider a closed curve $\gamma\subset B(z_0,\delta) \setminus \{z_0\}$ winding once around $z_0$; this implies that there exists a component $U$ of $\widehat{\C}\setminus \gamma$ containing $z_0$ but not $\infty$.  Since $\Gamma$ is connected and intersects the sets $U$ and $\widehat{\C} \setminus \br U$, we must have $\Gamma \cap  \gamma\neq \emptyset$. In fact, there exists $t>t_N$ such that $z\coloneqq f^{-1}(\Gamma^*(t)) \in \gamma$. It follows that $f(z)=\Gamma^*(t)\in B(w_0,\eta)$, as desired.

For the general case, we suppose again that each of $B,B^*$ is either a single point or the unit disk by using the Schoenflies theorem. Then we use Moore's theorem \ref{Decomposition} as in the proof of Proposition \ref{proposition:topological2} and obtain a continuous map $h$ of the sphere fixing $\infty$ such that $h$ is the identity in $B$ and in a neighborhood of $\infty$, it is a homeomorphism from $\Omega$ onto the domain $h(\Omega)$ and maps each complementary component of $\Omega$ different from $B$ to a point. Note that the boundary components of $\Omega$ and $h(\Omega)$ are in correspondence by Proposition \ref{proposition:topological} and therefore the complementary components of $h(\Omega)$ are single points, except possibly for $B$. We similarly obtain a continuous map $g$ from $\widehat{\C}$ onto itself that is a homeomorphism of $\Omega^*$ onto $g(\Omega^*)$, fixing $B^*$ and $\infty$ and mapping each complementary component of $\Omega^*$ that is different from $B^*$ to a point. Then the composition $\widetilde f=g\circ f\circ h^{-1}$ is a homeomorphism from $h(\Omega)$ onto $g(\Omega^*)$. The complementary components of $h(\Omega), g(\Omega^*)$ that are different from $B,B^*$ form  a totally disconnected set that is rel.\ closed in $\widehat{\C}\setminus B,\,\widehat{\C}\setminus B^*$, respectively. By Lemma \ref{lemma:extension} we conclude that $\widetilde{f}$ extends to a homeomorphism of $\widehat{\C}\setminus B$ onto $\widehat{\C}\setminus B^*$. The continuity of $h$ and $g$ implies that $w_0\in \clu(\widetilde f,z_0)$.

We fix $\eta>0$ and let $\gamma \subset B(z_0,\delta)$ be a closed curve as in the statement, winding once around $z_0$, where $\delta>0$ is to be determined. The following lemma implies that $h(\gamma)\subset \C\setminus \{z_0\}$ is a closed curve that winds once around $z_0$.
\begin{lemma}\label{lemma:winding}
Suppose that $h\colon \widehat{\C}\to \widehat{\C}$ is a continuous map such that $h^{-1}(z_0)=\{z_0\}$ and $h^{-1}(\infty)=\{\infty\}$ (in particular, the map $h$ fixes the points $z_0$ and $\infty$). Moreover, suppose that $h$ is equal to the identity in a neighborhood of $\infty$. If a closed curve $\gamma\subset \widehat{\C}\setminus \{z_0,\infty\}$ winds once around $z_0$, then $h(\gamma)$ also winds once around $z_0$.
\end{lemma}

Now we use the first case of the proof, applied to the homeomorphism $\widetilde f\colon \widehat{\C}\setminus B\to \widehat{\C}\setminus B^*$. It follows that for each $\eta'>0$ there exists $\delta'>0$ such that if $h(\gamma)\subset B(z_0,\delta')$, then there exists a point $z'\in h(\gamma)\setminus B$ such that $|\widetilde f(z')-w_0|<\eta'$. Note that if $\delta$ is sufficiently small, then by the continuity of $h$ we have $h(\gamma)\subset B(z_0,\delta')$. We remark that $z'$ does not lie necessarily in $h(\Omega)$. However, by the continuity of $\widetilde f$, we may find a point $h(z)\in h(\gamma)\cap h(\Omega)$ near $z'$ such that $|g(f(z))-w_0|=|\widetilde f(h(z))-w_0|<\eta'$. It remains to show that for each given $\eta>0$ we can choose a small $\eta'>0$ so that the above inequality implies that $|f(z)-w_0|<\eta$.

For the latter, it suffices to have that for each $\eta>0$ one can choose $\eta'>0$ such that if $|w-w_0|<\eta'$ for $w\in g(\Omega^*)$ then $|g^{-1}(w)-w_0| <\eta$. This follows immediately from Lemma \ref{lemma:clusterequality}, where instead of $f$ and $h$ one uses the identity $\id$ and  $g$.
\end{proof}

\begin{proof}[Proof of Lemma \ref{lemma:winding}]
We argue using homotopy. Namely, there exists a homotopy $\gamma_t \subset \widehat{\C}\setminus \{z_0,\infty\}$ such that $\gamma_0=\gamma$ and $\gamma_1$ is contained in a neighborhood of $\infty$, where $h$ is the identity. The winding number is invariant under homotopy (see \cite[Theorem 4.12, p.~90]{Burckel}), hence $\gamma_1$ still winds once around $z_0$. Since $h$ is the identity on $\gamma_1$, $h(\gamma_1)$ also winds once around $z_0$. Now, using the homotopy $h(\gamma_t)\subset \widehat{\C} \setminus \{z_0,\infty\}$ we  see that $h(\gamma_0)=h(\gamma)$ winds once around $z_0$.
\end{proof}

\section{Circles map to circles}
\label{sec4}
Our goal in this section is to show that a conformal map from a circle domain satisfying the quasihyperbolic condition of Theorem \ref{mainthm1} onto another circle domain cannot ``squeeze" a boundary circle to a point. This will be proved in Lemma \ref{lemma:circles to circles}.

\subsection{Fatness}\label{sec:fatness}

Before stating the lemma, we will need the notion of fatness of a set. A {measurable} set $B\subset \C$ is $c$-\textit{fat} for some constant $c>0$ if
\begin{equation}
\label{Definition:Fat}
\area (B \cap B(z,r)) \geq c r^2
\end{equation}
for all $z \in B$ and $0<r \leq \operatorname{diam}(B)$. A collection of sets is \textit{uniformly fat} if there exists a uniform $c>0$ such that each of the sets in the collection is $c$-fat. We also allow points to be considered fat (for any $c>0$). Note that circular disks in the plane are uniformly fat. The most important consequences of fatness that we will use repeatedly are the following:

\begin{enumerate}[(F1)]
\item Suppose that $B\subset \C$ is a $c$-fat, closed, connected set, and assume it intersects two concentric circles $\partial B(z,r), \partial B(z,R)$ with $0<r<R$. Then there exists a constant $C>0$ depending only on $c$ such that
\begin{align*}
\area(B\cap (B(z,R)\setminus B(z,r))) \geq C(R-r)^2.
\end{align*}
To see that, note that by the connectedness of $B$ there exists a point $y\in B\cap \partial B(z,(r+R)/2)$. Then $B(y, (R-r)/2) \subset B(z,R)\setminus B(z,r)$, {and} so
\begin{align*}
\area(B\cap (B(z,R)\setminus B(z,r)))  \geq c \frac{(R-r)^2}{4}.
\end{align*}

\item For a ball $B(z,r)$ and {a connected set} $B$ as above define
$$
d_r(B)\coloneqq\mathcal H^1(\{s\in [0,r]: B\cap \partial B(z,s)\neq \emptyset\}).
$$
Then (F1) implies that
\begin{align*}
\area(B\cap B(z,r)) \gtrsim d_r(B)^2,
\end{align*}
where the implicit constant depends only on $c$.

\item If $B$ is as above and $B\subset B(z,r)$, then $d_r(B)\simeq \diam(B)$, with implicit constants depending only on $c$. Indeed, trivially we have $d_r(B)\leq \diam(B)$, and also the fatness implies that $\area(B)\simeq \diam(B)^2$ with implicit constant depending only on $c$. On the other hand, since $B\subset B(z,r)$, the area of $B$ can also be bounded from above by a multiple of $\diam(B)\cdot d_r(B)$. Hence, $\diam(B)^2\lesssim \diam(B)d_r(B)$, which yields the conclusion.

\item Fatness is invariant under bi-Lipschitz maps. Namely, if $B$ is $c$-fat and $T\colon B\to B^*$ is $L$-bi-Lipschitz, then $B^*$ is $c'$-fat for a constant $c'$ depending only on $c$ and $L$. Moreover, fatness is invariant under scalings: if $s>0$ and $B$ is $c$-fat, then $sB=\{sx: x\in B\}$ is also $c$-fat. Combining these two facts with Koebe's distortion theorem (Lemma \ref{quasi:koebe distortion}), we obtain that fatness is invariant under conformal maps \textit{in sufficiently small scales}:

\begin{flushleft}
Let $f$ be a conformal map on a domain $D\subsetneq \C$ and consider a point $z_0\in D$, $0<r\leq \dist(z_0,\partial D)$, and $0<c_0<1$. If $B\subset B(z_0,c_0r)$ is a $c$-fat set, then $f(B)$ is a $c'$-fat set for a constant $c'$ depending only on $c$ and $c_0$.
\end{flushleft}

\end{enumerate}

A straightforward consequence of fatness of the complementary components of a domain is the following:

\begin{lemma}\label{lemma:countably many components}
If the complementary components of a domain $\Omega^*$ with $\infty\in \Omega^*$ are uniformly fat, then for each $\varepsilon>0$ there exist at most finitely many components $B^*$ of $\widehat{\C}\setminus \Omega^*$ with $\diam(B^*)>\varepsilon$. In particular, at most countably many complementary components of $\Omega^*$ can be non-degenerate.
\end{lemma}
\begin{proof}
The components $B^*$ of $\C\setminus \Omega^*$ are disjoint and they satisfy $\area(B^*)\simeq \diam(B^*)^2$ by fatness. Comparing the sum of the areas of the components of $\widehat{\C}\setminus \Omega^*$ to the area of a big ball $B(0,R)$ containing them gives the desired conclusion.
\end{proof}

\subsection{Circles map to circles}

The next lemma is the heart of this section.

\begin{lemma}\label{lemma:circles to circles}
Let $\Omega$ be a circle domain with $\infty\in \Omega$ and let $f$ be a conformal map from $\Omega$ onto another domain $\Omega^{*}$ with $f(\infty)=\infty \in \Omega^*$. Suppose that $\Omega$ satisfies the quasihyperbolic condition and that the complementary components of $\Omega^{*}$ are uniformly fat.

If $E\subset \partial \Omega$ is a non-degenerate continuum, then $\operatorname{Clu}(f,E)\subset \partial \Omega^*$ cannot be a single point. In particular, $f^{*}$ cannot map a boundary circle of $\Omega$ onto a single point boundary component of $\Omega^*$.
\end{lemma}

Recall that $\Omega$ satisfies the quasihyperbolic condition if there exists a ball $B(0,R)$ containing all complementary components of $\Omega$ and a point $x_0\in D\coloneqq B(0,R)\cap \Omega$ such that the  {inequality}
\begin{align*}
\int_D k(x,x_0)^2 \,dx<\infty
\end{align*}
holds.

\begin{proof}[Proof of Lemma \ref{lemma:circles to circles}]

We argue by contradiction, assuming that there exists a disk $\Sigma$ in the complement of $\Omega$ and an arc $E\subset \partial \Sigma$ such that $\clu(f,E)$ is a single point. Since the quasihyperbolic condition is invariant under translations and scalings (see Remark \ref{quasi:invariant}), we may assume that $\Sigma$ is the unit disk $\br \D$, and by postcomposing with a translation, we may also assume that $\clu(f,E)$ is the point $0\in \partial \Sigma^* \subset \partial \Omega^*$, where $\Sigma^*=f^*(\Sigma)$. We write $E=\{e^{i\theta}\colon \theta\in [\theta_1,\theta_2] \}$. We fix a small $r>0$, and let $H^*$ be the union of the complementary components of $\Omega^*$ intersecting $\br B(0,r)$, excluding $\Sigma^*$. We also define $H=(f^{*})^{-1} (H^*)$, $W^*=B(0,r)\cap \Omega^*$, and $W=f^{-1}(W^*)$. Here, $H$ is the union of the complementary components of $\Omega$ corresponding to components of $H^*$. However, to avoid introducing new notation, $H$ and $H^*$ will also be used to denote the corresponding collections of complementary components of $\Omega$ and $\Omega^*$, respectively.

We fix $\theta\in [\theta_1,\theta_2]$ and consider a ray $\gamma_\theta(t)=te^{i\theta}$, $1\leq t\leq t_0$, where $t_0>1$ is the first exit time of $\gamma_\theta$ from $W\cup H$, i.e., the first time $t_0>1$ such that $\gamma_\theta(t_0)$ intersects the closed set $\partial (W\cup H)$. If such a time does not exist, then one can find a sequence of points $z_n= \gamma_{\theta}(t_n)$, $t_n\searrow 1$, such that $z_n\notin W\cup H$ for all $n\in \N$. The diameters of the complementary components of $\Omega$ and $\Omega^*$ converge to $0$ by Lemma \ref{lemma:countably many components}. Using Lemma \ref{lemma:convergenceCluster} we may pass to a subsequence still denoted by $z_n$ such that either $z_n\in \Omega$ for all $n\in \N$ and $f(z_n) \to 0=\clu(f,E)$, or $z_n \in B_n \subset \widehat{\C}\setminus \Omega$ and $B_n^*=f^*(B_n) \to 0$. In the first case we have eventually $f(z_n)\in W^*$, so $z_n\in W$, a contradiction. In the second case we have eventually $B_n^*\subset H^*$, so $z_n\in  B_n \subset H$, which is again a contradiction.

We remark that $\gamma_\theta(t_0)$ either lies in $\Omega$, in which case we have $f(\gamma_\theta(t_0)) \in \partial B(0,r)$, or it lies in a complementary component $B \subset  {\widehat{\C}} \setminus \Omega$ such that $B^* \cap \partial B(0,r)\neq \emptyset$. Indeed, if $\gamma_\theta(t_0)$ lies in a complementary component $B$ with $B^*\subset B(0,r)$ then by (PT1) there exists a Jordan region $U^*$ with $B^* \subset U^*\subset B(0,r)$ and $\partial U^* \subset \Omega^*$. The preimage $\partial U\subset W$ of $\partial U^*$ bounds a Jordan region $U\subset W\cup H$ containing $B$; see Lemma \ref{lemma:surround}. Then $\gamma_\theta(t_0)$ would lie in the interior of $W\cup H$ and this contradicts the choice of $t_0$.

Now, the idea is to obtain some estimates based on consideration of the set
\begin{align*}
f(\gamma_\theta \cap W)\cup \bigcup_{\substack{B\in H \\ B\cap \gamma_\theta\neq \emptyset}} B^* \cup \{0\}
\end{align*}
``connecting" $0$ to $\partial B(0,r)$. The main difficulty however is that the set $\{B\in H : B\cap \gamma_\theta\neq \emptyset\}$ might be uncountable, and this makes estimates impossible. For this reason, we use the transboundary chain provided by Corollary \ref{quasi:detours:cor}. Specifically, if $\gamma_\theta(t_0)\in \Omega$, then we can obtain a transboundary chain consisting of paths $\gamma_1,\dots,\gamma_m$ and components $B_1,\dots,B_{m-1} \subset \widehat{\C}\setminus \Omega$ connecting the endpoints of $\gamma_\theta$. The paths $\gamma_i$ are contained in $\Omega$, except possibly for their endpoints, and also have several other important properties that we are going to use. If $\gamma_\theta(t_0)\in B_m$ for some component $B_m$ of $\widehat{\C}\setminus \Omega$, then we use Remark \ref{quasi:detours:cor:rem} instead to obtain a suitable transboundary chain of the form  $(\gamma_1,\dots,\gamma_m,B_m)$. Based on the properties of this transboundary chain, we claim:

\begin{claim}\label{lemma:cirlces claim}
The transboundary chain $\mathcal{C}=(\gamma_1, B_1, \dots, \gamma_{m-1}, B_{m-1}, \gamma_{m})$ given by Corollary \ref{quasi:detours:cor} (or $\mathcal{C}=(\gamma_1,\dots,\gamma_m,B_m)$ given by Remark \ref{quasi:detours:cor:rem}) satisfies

\begin{align*}
r\lesssim \sum_{\gamma_i \in \mathcal{C}} \sum_{\substack{Q\in \mathcal W(D)\\Q\cap  \gamma_i\neq \emptyset}} \ell(Q) \mint{-}_Q |f'| + \sum_{B_i \in \mathcal{C}} d_r(B_i^*),
\end{align*}
where
$$d_r(B_i^*)\coloneqq \mathcal H^1(\{s\in [0,r]:B_i^*\cap \partial B(0,s)\neq \emptyset\})$$
is the radial diameter of $B_i^*$.
The constants in the above inequality are uniform and do not depend on $r$, $\theta$, or the transboundary chain $\mathcal{C}$ given by Corollary \ref{quasi:detours:cor}.

\end{claim}

We shall use the claim now and prove it later. By condition (v) of Corollary \ref{quasi:detours:cor}, for a fixed $\varepsilon>0$, we may consider a transboundary chain $\mathcal{C}$ obtained from $\gamma_\theta$ such that if $Q\in \mathcal W(D)$ is a Whitney cube with $Q\cap \gamma_\theta=\emptyset$ but $Q\cap \gamma_i\neq \emptyset$ for some $i\in\{1,\dots,m\}$, then $\ell(Q)\leq \varepsilon$ and $SH(Q)\cap \gamma_\theta \neq \emptyset$. Moreover, condition (iv) asserts that the paths $\gamma_i$, $i=1,\dots,m$, intersect disjoint sets of Whitney cubes, and condition (ii) implies that each component $B_i\in \mathcal{C}$ is intersected by $\gamma_\theta$. Using Claim \ref{lemma:cirlces claim} {as well as} the above properties we obtain
\begin{align}\label{lemma:circles:inequality}
r\lesssim \sum_{\substack{Q\cap \br W\neq \emptyset\\Q\cap  \gamma_\theta\neq \emptyset}} \ell(Q) \mint{-}_Q |f'|+ \sum_{\substack{SH(Q)\cap \gamma_\theta\neq \emptyset\\ \ell(Q)\leq \varepsilon}} \ell(Q) \mint{-}_Q |f'|+\sum_{\substack{B\in H \\ B\cap  \gamma_\theta\neq \emptyset}} d_r(B^*).
\end{align}
Note that the last sum {contains at most countably many non-zero terms}, in view of Lemma \ref{lemma:countably many components}. The above inequality persists, if $\gamma_\theta$ denotes the full ray (instead of the truncated one) from $0$ to $\infty$.

For a set $A\subset \C$ and $\theta\in [0,2\pi]$ we define $\x_{A\cap \gamma_\theta}=1$ if $A\cap \gamma_\theta\neq \emptyset$ and otherwise $\x_{A\cap \gamma_\theta}=0$. If the set $A$ is compact, then the function $\theta \mapsto \x_{A\cap \gamma_\theta}$ is upper semi-continuous, and thus measurable. Note that the cubes $Q$, the complementary components $B$, and the shadows $SH(Q)$ are compact; see Lemma \ref{quasi:shadow:compact}. Therefore, the functions $\theta\mapsto \x_{Q\cap \gamma_\theta}$, $\theta\mapsto \x_{B\cap \gamma_\theta}$, and $\theta\mapsto \x_{SH(Q)\cap \gamma_\theta}$ are measurable.

Upon integrating \eqref{lemma:circles:inequality} over $[\theta_1,\theta_2]$ and applying Fubini's theorem, we obtain:
\begin{align}\label{lemma:circles A123}
\begin{aligned}
(\theta_2-\theta_1) r &\lesssim \sum_{Q\cap \br W\neq \emptyset} \ell(Q)\mint{-}_Q |f'| \int_{0}^{2\pi} \x_{Q\cap \gamma_\theta} \, d\theta \\
&\quad \quad + \sum_{\substack{Q\in \mathcal W(D)\\ \ell(Q)\leq \varepsilon}} \ell(Q)\mint{-}_{Q}|f'| \int_{0}^{2\pi} \x_{SH(Q)\cap \gamma_\theta} \, d\theta\\
&\quad\quad + \sum_{B\in H} d_r(B^*) \int_{0}^{2\pi} \x_{B\cap \gamma_\theta} \,d\theta\\
&\eqqcolon A_1+A_2+A_3.
\end{aligned}
\end{align}
We now treat each of the terms separately.

For $A_1$ note that
$$\int_{0}^{2\pi} \x_{Q\cap \gamma_\theta} \, d\theta \lesssim \diam(Q)\simeq \ell(Q),$$
since $Q\cap \br \D =\emptyset$. Thus,
\begin{align*}
A_1&\lesssim \sum_{Q\cap \br W\neq \emptyset} \int_Q |f'| = \int_{\bigcup_{Q\cap \br W\neq \emptyset} Q}|f'| \\
&\lesssim \area\left(\bigcup_{Q\cap \br W\neq \emptyset}Q \right)^{1/2}\cdot \left(\int_{\bigcup_{Q\cap \br W\neq \emptyset} Q}|f'|^2  \right)^{1/2}\\
&\simeq\area\left(\bigcup_{Q\cap \br W\neq \emptyset}Q \right)^{1/2}\cdot \area\left( \bigcup_{Q\cap \br W\neq \emptyset} f(Q)\right)^{1/2},
\end{align*}
since $f$ is conformal, and $|f'|^2$ is the Jacobian of $f$. As $r\to 0$, we have $W^*\to \{0\}$, and thus $W$ is contained in arbitrarily small neighborhoods of $\br \D$. This implies that the first term above is $o(1)$. Now we treat the second term. By Lemma \ref{quasi:koebe}, we have $\diam(f(Q))\simeq \dist(f(Q),\partial \Omega^*)$. On the other hand, if $Q\cap \br W\neq \emptyset$, then $f(Q)$ intersects $\br B(0,r)$, so $\dist(f(Q),\partial \Omega^*) \leq r$ because $0\in \partial \Sigma^*\subset  \partial \Omega^*$. It follows that $\diam(f(Q))\lesssim r$, and thus $f(Q)\subset B(0,cr)$ for a uniform constant $c>0$, whenever $Q\cap \br W\neq \emptyset$. We therefore obtain
\begin{align*}
\area\left( \bigcup_{Q\cap \br W\neq \emptyset} f(Q)\right) \leq \area(B(0,cr)) \simeq r^2.
\end{align*}
Summarizing, we have
\begin{align}\label{lemma:circles A_1}
A_1= o(r).
\end{align}

Next, we treat $A_2$. Exactly as in the computation for $A_1$, note that
\begin{align*}
 \int_{0}^{2\pi}  \x_{SH(Q)\cap \gamma_\theta}\, d\theta \lesssim \diam(SH(Q)) =s(Q),
\end{align*}
because the shadows $SH(Q)\subset \partial \Omega$ lie outside $\D$. Therefore,
\begin{align*}
A_2&\lesssim \sum_{\substack{Q\in  \mathcal W(D)\\ \ell(Q)\leq \varepsilon}} \ell(Q)\left(\mint{-}_{Q}|f'|^2\right)^{1/2} s(Q) \\
&\lesssim \left(\int_{\bigcup_{ \ell(Q)\leq \varepsilon}Q} |f'|^2 \right)^{1/2} \left(\sum_{\substack{Q\in \mathcal W(D)\\ \ell(Q)\leq \varepsilon}} s(Q)^2\right)^{1/2}.
\end{align*}
Recall that $r>0$ was fixed, and $\varepsilon>0$ was arbitrary. As in the computation for $A_1$, the first term represents the area of a subset of $f(D)$, and the latter is bounded. The second term converges to $0$ as $\varepsilon\to 0$ by Lemma \ref{quasi:shadow sum} and the quasihyperbolic condition. Hence, $A_2\to 0$ as $\varepsilon\to 0$.

Finally, we compute a bound for the term $A_3$. As before, the integral term is bounded by $\diam(B)$, so
\begin{align*}
A_3\lesssim \sum_{B\in H} d_r(B^*) \diam(B)\leq \left(\sum_{B^*\in H^*} d_r(B^*)^2\right)^{1/2}\left(\sum_{B\in H} \diam(B)^2\right)^{1/2}.
\end{align*}
Using  property (F2) from Subsection \ref{sec:fatness} we obtain $d_r(B^*)^2\lesssim \area(B^*\cap B(0,r))$. Therefore, the first sum is bounded by $\area(B(0,r))^{1/2}\simeq r$. Since each $B\in H$ is a circle or a point, we trivially have $\diam(B)^2\simeq \area(B)$. Therefore, the second term is comparable to
\begin{align*}
\area\left( \bigcup_{B\in H} B \right)^{1/2}.
\end{align*}
As $r\to 0$, all components $B\in H$ are contained in arbitrarily small neighborhoods of $\partial \D$; this follows from Lemma \ref{lemma:convergenceCluster} applied to $f^{-1}$. Hence the above area term is $o(1)$. Summarizing, $A_3= o(r)$.

Therefore, by \eqref{lemma:circles A123} and \eqref{lemma:circles A_1}, the vanishing of $A_2$, and the preceding paragraph we have
\begin{align*}
r\lesssim A_1+A_3 = o(r),
\end{align*}
a contradiction, {since all the implicit multiplicative constants are independent of $r$}.
\end{proof}

\begin{proof}[Proof of Claim \ref{lemma:cirlces claim}]
The proof will be based on properties (i), (iii), and (vi) in Corollary \ref{quasi:detours:cor}. That is, $\mathcal C$ is a transboundary chain connecting $\gamma_\theta(1)$ to $\gamma_{\theta}(t_0)$ and each of the paths $\gamma_1,\dots,\gamma_m$ of the chain $\mathcal C$ is piecewise linear. We will split the proof in two cases: $\gamma_\theta(t_0)\in \Omega$, in which case $f(\gamma_\theta(t_0)) \in \partial B(0,r)$, and $\gamma_\theta(t_0)$ lies in a complementary component $B_m \subset  {\widehat{\C}} \setminus \Omega$ such that $B_m^* \cap \partial B(0,r)\neq \emptyset$.

Suppose first that $\gamma_\theta(t_0)\in \Omega$, so $f(\gamma_\theta(t_0)) \in \partial B(0,r)$. Then the chain $\mathcal C$ is of the form $(\gamma_1,B_1,\dots,\gamma_{m-1},B_{m-1},\gamma_m)$ and by (i), $\gamma_1$ starts at $a_1=\gamma_\theta(1) \in E$ and $\gamma_m$ terminates at $b_m=\gamma_{\theta}(t_0)$. Let us assume that $B_i^*\cap \partial B(0,r)=\emptyset$ for all $i=1,\dots,m-1$. In this case, $B_i^*\subset B(0,r)$ by Lemma \ref{lemma:surround}, so using property (F3) from Subsection \ref{sec:fatness} we deduce that $\diam(B_i^*) \simeq d_r(B_i^*)$  for $i=1,\dots,m-1$. By Lemma \ref{lemma:connected} the set
\begin{align*}
\mathcal K \coloneqq \bigcup_{\gamma_i\in \mathcal C} \br{f(\gamma_i)} \cup \bigcup_{B_i\in \mathcal C} B_i^*
\end{align*}
is a a continuum joining $0 \in \clu(f,a_1)\subset \clu(f,E)$ to $f(b_m)\in \partial B(0,r)$. Therefore,
\begin{align*}
r\leq \diam(\mathcal K) &\leq \sum_{\gamma_i\in \mathcal C} \diam (f(\gamma_i))+ \sum_{B_i\in \mathcal C} \diam(B_i^*)  \\
&\lesssim \sum_{\gamma_i\in \mathcal C} \diam (f(\gamma_i))+ \sum_{B_i\in \mathcal C} d_r(B_i^*).
\end{align*}
Note that $Q\cap\gamma_i$, if non-empty, is contained in the union of two line segments by Corollary \ref{quasi:detours:cor}(vi). Combining this with Lemma \ref{quasi:koebe}, we have
\begin{align*}
\diam (f(\gamma_i)) \leq \int_{\gamma_i} |f'| ds  \lesssim \sum_{\substack{ Q\in \mathcal W(D)\\ Q\cap \gamma_i\neq \emptyset}}  \ell(Q) \max_{Q} |f'| \simeq \sum_{\substack{ Q\in \mathcal W(D)\\ Q\cap \gamma_i\neq \emptyset}}  \ell(Q) \mint{-}_{Q}|f'|.
\end{align*}
This completes the proof in this case.

If $B_k^* \cap \partial B(0,r)\neq \emptyset$ for some $k=1,\dots,m-1$, we assume that $k$ is the first such index, and consider the transboundary chain $\mathcal C'= (\gamma_1,B_1,\dots,\gamma_{k-1}, B_{k-1},\gamma_k)\subset \mathcal C$. Note that $\br{f(\gamma_k)}$ intersects the set $B_k^*$, since $\gamma_k$ intersects $B_k$. It follows from Lemma \ref{lemma:connected} that the set
\begin{align*}
\mathcal K' \coloneqq \bigcup_{\gamma_i\in \mathcal C'} \br{f(\gamma_i)} \cup \bigcup_{B_i\in \mathcal C'} B_i^*
\end{align*}
is a continuum joining $0$ to a point of a circle $\partial B(0,r_1)$, $0<r_1\leq r$, where $\partial B(0,r_1)\cap B_k^*\neq \emptyset$. Note that $d_r(B_k^*)\geq r-r_1$. Since $B_i^*\subset B(0,r)$ for all $i=1,\dots,k-1$, we can argue for $\mathcal K'$ as in the previous case of the proof and obtain
\begin{align*}
r = r_1+(r-r_1)\lesssim \sum_{\gamma_i\in \mathcal C'} \sum_{\substack{ Q\in \mathcal W(D)\\ Q\cap \gamma_i\neq \emptyset}}  \ell(Q) \mint{-}_{Q}|f'|+ \sum_{B_i\in \mathcal C'} d_r(B_i^*)+ d_r(B_k^*).
\end{align*}
This is the desired estimate.

Finally, we discuss the second case that $\mathcal C=(\gamma_1,\dots,\gamma_m,B_m)$, where $B_m\cap \partial B(0,r)\neq \emptyset$. This is done exactly as the previous paragraph, by considering a chain $\mathcal C'=(\gamma_1,B_1,\dots,B_{k-1},\gamma_k)\subset \mathcal C$ such that $B_i^*\subset B(0,r)$ for all $i=1,\dots, k-1$ and $B_{k}^* \cap \partial B(0,r)\neq \emptyset$.
\end{proof}

\section{Points map to points}
\label{sec5}

In this section we prove that a conformal map from a circle domain satisfying the quasihyperbolic condition onto another circle domain must map point boundary components to point boundary  components. In fact, we prove a more general result.

\begin{lemma}\label{lemma:points}
Let $\Omega$ be a circle domain with $\infty\in \Omega$ and let $f$ be a conformal map from $\Omega$ onto another domain $\Omega^{*}$ with $f(\infty) = \infty \in \Omega^*$. Suppose that $\Omega$ satisfies the quasihyperbolic condition and that the complementary components of $\Omega^{*}$ are uniformly fat (closed) Jordan regions or points.

Then for each $z_0\in \partial \Omega$ the cluster set $\clu(f,z_0)\subset \partial \Omega^*$ cannot be a non-degenerate continuum. In particular, $f^{*}$ cannot map a point boundary component of $\Omega$ onto a non-degenerate component of $\partial \Omega^*$.
\end{lemma}

The proof is very similar to the proof of Lemma \ref{lemma:circles to circles}, so we omit some of the details.

\begin{proof}
Suppose that the cluster set of the point $ 0 \in \Sigma \subset \partial \Omega$ is a non-degenerate continuum $E^* \subset \Sigma^*$, with $\diam(E^*)=1$, after rescaling.  We fix a small $r>0$, and consider the annulus $A_r= A(0;r/2,r)\coloneqq \{z: r/2<|z|<r\}$. We denote by $H$ the union of the complementary components of $\Omega$ intersecting $\overline{A_r}$ excluding $\Sigma$, and by $H^*$ the union of the corresponding components of $\widehat{\C}\setminus\Omega^*$. Here $H$ and $H^*$ will also be used to denote the corresponding collections of components. We also define $W=A_r\cap \Omega$ and $W^*=f(A_r\cap \Omega)$.

We fix $\rho\in [r/2,r]$, and consider a circle $\gamma_\rho(t)=\rho e^{it}$, $t\in [0,2\pi]$. By Lemma \ref{lemma:crosscut}, if $r$ is sufficiently small, then there exist points $z_1,z_2\in \gamma_\rho \cap \Omega$ such that
$$|f(z_1)-f(z_2)|\geq \diam (E^*)/2=1/2.$$
We now apply, for a fixed $\varepsilon>0$, Corollary \ref{quasi:detours:cor} to obtain from $\gamma_\rho$ a transboundary chain $\mathcal C$ connecting $z_1$ and $z_2$. Using Lemma \ref{lemma:connected}, we obtain the analog of Claim \ref{lemma:cirlces claim}:
\begin{align*}
1/2 \leq |f(z_1)-f(z_2)|\lesssim \sum_{\gamma_i \in \mathcal{C}} \sum_{\substack{Q\in \mathcal W(D)\\Q\cap  \gamma_i\neq \emptyset}} \ell(Q) \mint{-}_Q |f'| + \sum_{B_i \in \mathcal{C}} \diam(B_i^*).
\end{align*}
The implicit constant is independent of $r,\rho,\varepsilon$ and of the transboundary chain $\mathcal C$ obtained from $\gamma_\rho$.

As in the proof of Lemma \ref{lemma:circles to circles}, the properties of the chain $\mathcal C$ from Corollary \ref{quasi:detours:cor} yield for each $\varepsilon>0$
\begin{align*}
1\lesssim \sum_{\substack{Q\in \mathcal W(D)\\Q\cap  \gamma_\rho\neq \emptyset}} \ell(Q) \mint{-}_Q |f'|+ \sum_{\substack{SH(Q)\cap \gamma_\rho\neq \emptyset\\ \ell(Q)\leq \varepsilon}} \ell(Q) \mint{-}_Q |f'|+\sum_{\substack{B\in H \\ B\cap  \gamma_\rho\neq \emptyset}} \diam(B^*).
\end{align*}
Now, we integrate over $\rho \in [r/2,r]$ and we obtain
\begin{align*}
r&\lesssim \sum_{Q\cap \br W\neq \emptyset} \ell(Q)\mint{-}_{Q} |f'| \ell(Q) + \sum_{\substack{Q\in \mathcal W(D)\\ \ell(Q)\leq \varepsilon}} \ell(Q)\mint{-}_{Q} |f'|s(Q)+ \sum_{B\in H} \diam(B^*)\delta_r(B)\\
&\eqqcolon A_1+A_2+A_3,
\end{align*}
where $\delta_r(B)\coloneqq \mathcal H^1(\{s\in [r/2,r]: B\cap \partial B(0,s)\neq \emptyset\})$.

The middle term $A_2$ vanishes as $\varepsilon\to 0$, {because of} the quasihyperbolic condition, exactly as in the proof of Lemma \ref{lemma:circles to circles}. The first term $A_1$ is bounded from above by
\begin{align*}
\area\left(\bigcup_{Q\cap \br W\neq \emptyset}Q \right)^{1/2}\cdot \area\left( \bigcup_{Q\cap \br W\neq \emptyset} f(Q)\right)^{1/2}.
\end{align*}
Note that the Whitney cubes intersecting $W=A_r\cap \Omega$ must have sidelength bounded by a constant multiple of $r$, since $0\in \partial \Omega$. Hence, the first factor is $O(r)$. The second factor  is $o(1)$ as $r\to 0$, since $W^*\to \partial \Sigma^*$ as $r\to 0$. Hence, $A_1=o(r)$. Finally, $A_3$ is bounded above by
\begin{align*}
\left(\sum_{B^*\in H^*} \diam(B^*)^2\right)^{1/2}\left(\sum_{B\in H} \delta_r(B)^2\right)^{1/2}.
\end{align*}
By the fatness condition, the first factor is bounded by a constant multiple of the area of the components $B^*\in H^*$. Since all these components are contained in arbitrarily small neighborhoods of $\Sigma^*$ as $r\to 0$ (by Lemma \ref{lemma:convergenceCluster}), it follows that the contribution here is $o(1)$. The second factor, by the fatness of disks and property (F1), is bounded by a constant multiple of the square root of the area of the annulus $A_r$, so the contribution is $O(r)$. Hence, $A_2=o(r)$.

Summarizing, $r\lesssim A_1+A_3= o(r)$, a contradiction.
\end{proof}

\section{Continuous extension}
\label{sec6}

We now have everything we need in order to prove homeomorphic extension to the boundary.

\begin{theorem}\label{thm:continuous}
Let $\Omega$ be a circle domain with $\infty\in \Omega$ and let $f$ be a conformal map from $\Omega$ onto another domain $\Omega^{*}$ with $f(\infty) = \infty \in \Omega^*$. Suppose that $\Omega$ satisfies the quasihyperbolic condition and that the complementary components of $\Omega^{*}$ are uniformly fat (closed) Jordan regions or points. Then $f$ extends to a homeomorphism from $\br \Omega$ onto $\overline{\Omega^*}$.
\end{theorem}

The proof is based on Lemmas \ref{lemma:circles to circles} and \ref{lemma:points}.

\begin{proof}
First we prove that $f$ extends continuously to $\br \Omega$. Let $\Sigma$ be a component of $\widehat{\C} \setminus \Omega$. If $\Sigma$ is a point, then $\Sigma^*=f^*(\Sigma)$ has to be a singleton, by Lemma \ref{lemma:points}. Hence, in this case $f$ extends continuously to $\Omega\cup \Sigma$. If $\Sigma$ is a disk, then by Lemma \ref{lemma:circles to circles} the component $\Sigma^*$ is  a Jordan region. Suppose that $f$ does not extend continuously to a point $p\in \partial \Sigma$. Then, as $\Omega\ni z\to p$, the images $f(z)$ have to accumulate in at least two distinct points of $\partial \Sigma^*$, from which we deduce that the cluster set $\operatorname{Clu}(f;p)$ is a non-degenerate continuum, by Proposition \ref{proposition:topological2}. This contradicts Lemma \ref{lemma:points}. Therefore, $f$ extends continuously to $\br \Omega$. Note that $f(\br \Omega)=\overline{\Omega^*}$, since any point of $\partial \Omega^*$ is the accumulation point of a sequence $f(z_n)$, where $\Omega\ni z_n \to \partial \Omega$.

Now, we wish to show that $f$ is injective on $\br \Omega$. Since $\br \Omega$ is compact, it will then follow that we have a homeomorphism from $\br \Omega$ onto $\overline{\Omega^*}$, as desired. Note that each component of $\partial \Omega$ is mapped continuously onto a component of $\partial \Omega^*$, and the correspondence of the components is one-to-one, because $f^*$ is bijective by Proposition \ref{proposition:topological}. Since point boundary components are mapped to point boundary components, it follows that $f$ is injective there. Hence, it suffices to prove that $f$ is injective when restricted to a circle component $\partial \Sigma$ of $\partial \Omega$.

Suppose that this is not the case. Then there exist two distinct points on $\partial \Sigma$ that are mapped to a single point $p\in \partial \Sigma^*$. Since the complementary components of $\Omega^*$ are uniformly fat, their diameters converge to $0$ by Lemma \ref{lemma:countably many components}. We can therefore apply Proposition \ref{proposition:topological2} to $f^{-1}\colon \Omega^* \to \Omega$ to deduce that $E\coloneqq\operatorname{Clu}(f^{-1};p)$ is a non-degenerate continuum. By the continuity of $f$ on $\br\Omega$ it follows that $\operatorname{Clu}(f;E)=p$, contradicting Lemma \ref{lemma:circles to circles}.
\end{proof}

\section{Quasiconformal extension}
\label{sec7}

Let $\Omega$ be a circle domain with $\infty \in \Omega$, and suppose that $\Omega$ satisfies the quasihyperbolic condition. The goal of this section is to prove that there exists a uniform constant $K$ such that every conformal map of $\Omega$ onto another circle domain extends to a $K$-quasiconformal homeomorphism of $\RiemannSphere$.

\subsection{Homeomorphic extension by reflection}
\label{subsechom}

Let $f\colon \Omega \to \Omega^*$ be a conformal map of $\Omega$ onto another circle domain $\Omega^*$, and assume without loss of generality that $f(\infty)=\infty \in \Omega^*$. By Theorem \ref{thm:continuous}, the map $f$ extends to a homeomorphism of $\overline{\Omega}$ onto $\overline{\Omega^*}$, which we still denote by the letter $f$. Our goal now is to use repeated Schwarz reflections to extend $f$ to a homeomorphism of $\RiemannSphere$ that conjugates the Schottky groups of $\Omega$ and $\Omega^*$. First, we need some notation and definitions. The interested reader may also want to consult \cite[Section 3]{BKM}, which contains similar material.

Let $\{\gamma_j\}$ be the collection of disjoint circles in $\partial \Omega$, and for each $j$ denote by $R_j \colon  \RiemannSphere \to \RiemannSphere$ the reflection across the circle $\gamma_j$, i.e.,
$$R_j(z)\coloneqq a_j + \frac{r_j^2}{\overline{z}-\overline{a_j}},$$
where $a_j$ is the center and $r_j$ is the radius of the circle $\gamma_j$. Also, denote by $B_j$ the open disk bounded by the circle $\gamma_j$; we remark that the letter $B$ was used in previous sections to denote a (closed) component of $\widehat{\C}\setminus \Omega$, but here $B_j$ is open. Note that the disks $B_j$ have pairwise disjoint closures and that $\RiemannSphere \setminus \overline{\Omega} = \bigcup_j B_j$.

\begin{definition}
The \textit{Schottky group} $\Gamma(\Omega)$ is the free discrete group of M\"{o}bius and anti-M\"{o}bius transformations generated by the family of reflections $\{R_j\}$.
\end{definition}

Thus $\Gamma(\Omega)$ consists of the identity map and all transformations of the form $T=R_{i_1}\circ \cdots \circ R_{i_k}$, $k\geq 1$, where $i_j \neq i_{j+1}$ for $j=1,\dots,k-1$. If $T\in \Gamma(\Omega)$ is represented in this form, then we say that $T$ is written in \textit{reduced form} and the sequence of indices $i_1,\dots,i_{k}$ is also called reduced whenever consecutive indices are distinct.

With this reduced form, we define the \textit{length} of $T$ by $l(T)\coloneqq k$. The length of the identity map is defined to be zero.

A simple observation that we will use repeatedly is that if $T=R_{i_1}\circ \cdots \circ R_{i_k}$ is in reduced form, then $T(\Omega) \subset B_{i_1}$. In particular $T(\Omega) \cap \overline{\Omega} = \emptyset$, and the map $T$ cannot be equal to the identity. This implies that the representation of $T$ in reduced form is unique and thus $l(T)$ is well-defined. Indeed, suppose that
$$R_{i_1} \circ \cdots \circ R_{i_k} = R_{j_1} \circ \cdots \circ R_{j_l}$$
are distinct representations of $T$ in reduced form. Using the fact that each reflection $R_j$ is its own inverse, we can simplify to obtain $R_{m_1} \circ \cdots \circ R_{m_n} = \operatorname{id}$ for some $n \geq 1$, contradicting the above. See also \cite[Proposition 3.1]{BKM}.

We remark that for each $k\geq 0$ there are at most countably many $T\in \Gamma(\Omega)$ with $l(T)=k$.

Now, for each $k\geq 0$, consider the union of reflected domains
$$\Omega_k\coloneqq \bigcup_{l(T) \leq k} T(\Omega),$$
where the union is taken over all elements $T$ of the Schottky group $\Gamma(\Omega)$ with length $0 \leq l(T) \leq k$. We will need the following properties of the open sets $\Omega_k$.

\begin{lemma}\label{lemma:equations}

For each $k\geq 0$, we have

\begin{equation}
\label{equation:closure}
\overline{\Omega_k} = \bigcup_{l(T) \leq k} T(\overline{\Omega}),
\end{equation}
\begin{align}\label{equation:boundaries}
\partial \Omega_k = \bigcup_{l(T) \leq k} T(\partial \Omega),
\end{align}
and
\begin{align}\label{equation:complement}
\begin{aligned}
\RiemannSphere \setminus \overline{\Omega_k} &= \{ R_{i_1} \circ \dots \circ R_{i_k}(B_{i_{k+1}}): i_j \neq i_{j+1}, \, j=1,\dots,k \} \\
&=\bigsqcup_{i_j\neq i_{j+1}} R_{i_1} \circ \dots \circ R_{i_k}(B_{i_{k+1}}),
\end{aligned}
\end{align}
which is understood as $\bigcup_j B_j$ when $k=0$. Moreover, for each reduced sequence $i_1,\dots,i_{k+2}$ we have
\begin{align}\label{equation:containment}
R_{i_1} \circ \dots \circ R_{i_{k+1}}(B_{i_{k+2}}) \subset \subset R_{i_1} \circ \dots \circ R_{i_k}(B_{i_{k+1}}),
\end{align}
which is understood as $R_{i_1}(B_{i_2})\subset\subset  B_{i_1}$ when $k=0$.
\end{lemma}
Here $A\subset \subset B$ means that the set $\br A$ is compact and is contained in $B$.

\begin{proof}
We first prove by induction on $k$ that the set $F_k\coloneqq \bigcup_{l(T) \leq k} T(\overline{\Omega})$ is closed. Note that $F_0 = \overline{\Omega}$, which is a closed set. Suppose that $F_k$ is closed for some $k \geq 0$, and let $w_n$ be a sequence in $F_{k+1}$ with $w_n \to w$. For each $n$, write $w_n = T_n(z_n)$, where $z_n \in \overline{\Omega}$ and $T_n \in \Gamma(\Omega)$ with $l(T_n) \leq k+1$. If $T_n$ is equal to the identity map for infinitely many $n$, then $w_n = z_n \in \overline{\Omega}$ for these values of $n$, which implies that $w \in \overline{\Omega} \subset F_{k+1}$. We can therefore assume that $T_n \neq \operatorname{id}$ for all $n$. Now, write $T_n = Q_n \circ S_n$, where $Q_n$ is the reflection across one of the boundary circles $\gamma_j$, say $C_n$, and $l(S_n) \leq k$. There are two possibilities to consider.

First, suppose that infinitely many of the reflections $Q_n$ are equal, say $Q_{n_j}=R$ for all $j$. This gives $w_{n_j} = T_{n_j}(z_{n_j}) =  R \circ S_{n_j}(z_{n_j})$, so that $R(w_{n_j}) = S_{n_j}(z_{n_j})$ for all $j$, since $R^{-1}=R$. But $S_{n_j}(z_{n_j}) \in F_k$ for all $j$ and $R(w_{n_j}) \to R(w)$ as $j \to \infty$, so that $R(w) \in F_k$, since $F_k$ is assumed to be closed. This implies that $w \in F_{k+1}$, as required.

In the second case, passing to a subsequence if necessary, we can assume that all the reflections $Q_n$ are distinct. The closed disks $D_n$ in $\RiemannSphere \setminus \Omega$ bounded by the circles $C_n$ are then pairwise disjoint and hence must have diameters converging to zero; see Lemma \ref{lemma:countably many components}. Now, note that $w_n \in D_n$ for all $n$, so that necessarily $w \in \partial \Omega$, since $w_n \to w$ and the sequence of disks $D_n$ can only accumulate at the boundary of $\Omega$. In particular, the limit point $w$ belongs to $F_{k+1}\supset \partial \Omega$.

In both cases, we get that $w \in F_{k+1}$, so that $F_{k+1}$ is closed, completing the proof by induction. Now, note that since $F_k$ is a closed set containing $\Omega_k$, we must have $\overline{\Omega_k} \subset F_k$. The other inclusion is trivial, and (\ref{equation:closure}) follows.
\\

To prove (\ref{equation:boundaries}), let $w \in \partial \Omega_k$, so that $w \in \overline{\Omega_k}$ but $w \notin \Omega_k$. By (\ref{equation:closure}), there exists $T \in \Gamma(\Omega)$ with $l(T) \leq k$ such that $w \in T(\overline{\Omega})$. Note that $w$ cannot belong to $T(\Omega)$, since $w\notin \Omega_k$. It follows that $w \in T(\partial \Omega)$. This shows that
$$\partial \Omega_k \subset \bigcup_{l(T) \leq k} T(\partial \Omega).$$
For the other inclusion, suppose that $T \in \Gamma(\Omega)$ with $l(T) \leq k$. Then $T(\partial \Omega) \subset T(\overline{\Omega}) \subset \overline{\Omega_k}$, by (\ref{equation:closure}). On the other hand, it is easy to see that for every $S \in \Gamma(\Omega)$ the set $T(\partial \Omega)$ cannot intersect $S(\Omega)$, including when $S=T$. {Indeed, suppose that $T(\partial \Omega)\cap S(\Omega)\neq \emptyset$ for some $S\in \Gamma(\Omega)$. This implies that $\partial \Omega \cap T^{-1}( S(\Omega))\neq \emptyset$. Note that $T^{-1}( S(\Omega))$ is either equal to $\Omega$ if $S=T$, or it lies in an open disk of $\widehat{\C}\setminus\br{\Omega}$. In both cases we obtain a contradiction.} It follows that $T(\partial \Omega) \cap \Omega_k=\emptyset$. Thus $T(\partial \Omega) \subset \partial \Omega_k$ for each $T \in \Gamma(\Omega)$ with $l(T) \leq k$, and taking unions gives
$$\bigcup_{l(T) \leq k} T(\partial \Omega) \subset \partial \Omega_k.$$
This completes the proof of (\ref{equation:boundaries}).
\\

Finally, to prove (\ref{equation:complement}), recall that $\RiemannSphere \setminus \overline{\Omega} = \bigcup_j B_j$, so that (\ref{equation:complement}) holds for $k=0$. Now, suppose that $k \geq 1$, and let $w \notin \overline{\Omega_k}$. Then in particular, we have that $w \notin \overline{\Omega}$, so that $w \in B_{i_1}$ for some $i_1$. Moreover, we have that $R_{i_1}(w) \notin \overline{\Omega}$, otherwise $w$ would lie in $\br {\Omega_1}$. Hence, there exists $i_2 \neq i_1$ such that $R_{i_1}(w) \in B_{i_2}$, i.e., $w \in R_{i_1}(B_{i_2})$. Repeating this process, we get a reduced sequence of indices $i_1, \dots, i_{k+1}$ such that $ w \in R_{i_1} \circ \dots \circ R_{i_k}(B_{i_{k+1}})$. This shows that

$$ \RiemannSphere \setminus \overline{\Omega_k} \subset \{ R_{i_1} \circ \dots \circ R_{i_k}(B_{i_{k+1}}): i_j \neq i_{j+1} ,\, j=1,\dots,k \}.$$
To prove the reverse inclusion, let $w = R_{i_1} \circ \dots \circ R_{i_k} (z)$ for $z \in B_{i_{k+1}}$, where $i_j \neq i_{j+1}$ for $j=1,\dots,k$. Suppose for a contradiction that $w \in \overline{\Omega_k}$. Then by (\ref{equation:closure}), we have that $w= S(z_0)$ for some $S \in \Gamma(\Omega)$ with $l(S) \leq k$ and some $z_0 \in \overline{\Omega}$. Writing $S=R_{j_1} \circ \dots \circ R_{j_l}$, $l \leq k$, in reduced form, we get
$$R_{i_1} \circ \dots \circ R_{i_k} (z) = R_{j_1} \circ \dots \circ R_{j_l} (z_0).$$
But the left-hand side belongs to $B_{i_1}$ while the right-hand side belongs to $\overline{B_{j_1}}$, so that $i_1=j_1$, since the disks $B_j$ have pairwise disjoint closures. Simplifying and repeating this argument, we get
$$R_{i_{l+1}} \circ \dots \circ R_{i_k}(z)=z_0,$$
if $l<k$, or $z=z_0$, if $l=k$. This is clearly impossible since $z \in B_{i_{k+1}}$ and $z_0 \in \overline{\Omega}$. It follows that
$$ \{ R_{i_1}\circ \dots \circ R_{i_k}(B_{i_{k+1}}): i_j \neq i_{j+1},\, j=1,\dots,k \} \subset \RiemannSphere \setminus \overline{\Omega_k},$$
as required. A similar argument shows that the disks $R_{i_1}\circ \dots \circ R_{i_k}(B_{i_{k+1}})$ are disjoint for distinct reduced sequences $i_1,\dots,i_{k+1}$. This completes the proof of (\ref{equation:complement}).

For \eqref{equation:containment} we note that $R_{i_{k+1}}(B_{i_{k+2}}) \subset\subset B_{i_{k+1}}$ since the circles $\gamma_{i_{k+1}},\gamma_{i_{k+2}} \subset \partial \Omega$ are disjoint. Applying the reflections $R_{i_k},\dots,R_{i_1}$ to both sides of the inclusions gives the result.
\end{proof}

By (\ref{equation:complement}), the complement of $\overline{\Omega_k}$ is the union of disjoint open disks, each obtained by $k$ reflections of a disk in $\RiemannSphere \setminus \overline{\Omega}$.

\begin{proposition}\label{disks:converge to 0}
The area of each disk in the complement of $\overline{\Omega_k}$ tends to zero as $k \to \infty$.
\end{proposition}

See also \cite[Lemma 3.3]{BKM}.

\begin{proof}
We first prove the result in the case where $\partial \Omega$ contains only finitely many circles.

A simple calculation shows that the absolute value of the Jacobian of $R_j$ is $r_j^4/|z-a_j|^4$, which is less than $r_j^4/d^4<1$ on $\bigcup_{k \neq j} B_k$, where $d>0$ is the infimum of the distances between $a_j$ and each $\gamma_k$, $k \neq j$. Now, since there are only finitely many circles $\gamma_j$, the ratio $r_j^4/d^4$ can be bounded uniformly away from $1$, independently of $j$, so that the area of each disk decreases by a definite amount after each reflection. This proves the result in this case. Our computation also shows that the reflection from the exterior to the interior of a circle is area-decreasing.

For the general case, suppose for a contradiction that there exists some $\varepsilon>0$ and disks $D_k \subset \RiemannSphere \setminus \overline{\Omega_k}$, $k \in \mathbb{N}$, such that for each $k$, $D_k$ has area larger than $\varepsilon$. By \eqref{equation:complement} each $D_k$ is necessarily obtained by reflections of a disk $B_{j_k}$ in $\RiemannSphere \setminus \br\Omega$ with area larger than $\varepsilon$ along finitely many circles, each of them bounding a disk with area also larger than $\varepsilon$. This is because these reflections are area-decreasing, by the first part of the proof. However, there are only finitely many disks in $\RiemannSphere \setminus \br \Omega$ with area bigger than $\varepsilon$, so $D_k$ is obtained by reflecting along the same finite family of circles for all $k\in \N$. The first part of the proof then gives a contradiction.
\end{proof}

We define $\Omega_\infty= \bigcup_{k=0}^\infty \br{\Omega_k}$. Then we have the following immediate consequence of Lemma \ref{lemma:equations} and Proposition \ref{disks:converge to 0}:
\begin{corollary}\label{corollary:OmegaInfinity}
For each $z\in \widehat{\C}\setminus \Omega_\infty$ there exist a unique sequence of indices $\{i_j\}_{j\in \N}$ with $i_j\neq i_{j+1}$ for $j\in \N$, and disks $D_0=B_{i_1}$, $D_k=R_{i_1}\circ \dots \circ R_{i_k}(B_{i_{k+1}})$, such that $D_{k+1}\subset\subset  D_k$ for $k\geq 0$, and $\{z\}= \bigcap_{k=0}^\infty D_k$.

Conversely, if $D_k$ is a sequence of disks as above, then $\bigcap_{k=0}^\infty D_k$ is a single point contained in $\widehat{\C}\setminus \Omega_\infty$.
\end{corollary}

Now, if $f\colon \Omega \to \Omega^*$ is a conformal map of $\Omega$ onto another circle domain $\Omega^*$ that extends to a homeomorphism of $\overline{\Omega}$ onto $\overline{\Omega^*}$, we denote by $\{\gamma_j^*\}$ the collection of boundary circles in $\partial \Omega^*$, enumerated in such a way that $f(\gamma_j)=\gamma_j^*$ for all $j$. We also denote by $B_j^*$ the open disk in $\RiemannSphere \setminus \overline{\Omega^*}$ bounded by the circle $\gamma_j^*$. Also, let $R_j^*$ be the reflection across the circle $\gamma_j^*$ and $\Gamma(\Omega^*)$ be the Schottky group of $\Omega^*$. We denote by $T^*$ the element of $\Gamma(\Omega^*)$ corresponding to $T \in \Gamma(\Omega)$ and observe that $(T^*)^{-1}=(T^{-1})^*$. Finally, we define the sets $\Omega_k^*$ by
$$\Omega_k^*\coloneqq \bigcup_{l(T^*) \leq k} T^*(\Omega^*).$$

The proof of the following extension result was sketched in \cite[Lemma 3.1]{SCH2}. We give a complete proof for the convenience of the reader.

\begin{lemma}
\label{HomeoExtension}
The map $f \colon  \overline{\Omega} \to \overline{\Omega^*}$ extends uniquely to a  homeomorphism $\tilde{f} \colon  \RiemannSphere \to \RiemannSphere$ that conjugates the Schottky groups of $\Omega$ and $\Omega^*$, i.e.,

$$T^* = \tilde{f} \circ T \circ  \tilde{f}^{-1} \qquad (T \in \Gamma(\Omega)).$$
\end{lemma}

\begin{proof}
First define $\tilde{f}$ on $\overline{\Omega_1}$ by
\begin{displaymath}
\tilde{f} \coloneqq  \left\{ \begin{array}{ll}
f & \textrm{on $ \overline{\Omega}$}\\
R_j^* \circ f \circ R_j & \textrm{on $R_j(\overline{\Omega})$.}\\
\end{array} \right.
\end{displaymath}
Next, extend $\tilde{f}$ to $\overline{\Omega_2}$ by defining
$$
\tilde{f} \coloneqq (R_k^* \circ R_j^*) \circ f \circ (R_j \circ R_k) \, \, \textrm{on $R_k \circ R_j(\overline{\Omega})$.}
$$

Repeating this process defines an extension $\tilde{f}$ of $f$ to $\Omega_\infty =  \bigcup_{k=0}^\infty \overline{\Omega_k}$. By construction, the map $\tilde{f}$ is a bijection that conjugates the Schottky groups $\Gamma(\Omega)$,$\Gamma(\Omega^*)$ and maps $\Omega_\infty$ onto $\Omega_\infty^*\coloneqq   \bigcup_{k=0}^\infty \overline{\Omega_k^*}$. Note that $\tilde{f}$ is conformal on $\bigcup_{k=0}^\infty \Omega_k$.

Now, we claim that $\tilde{f}$ is continuous on $\Omega_\infty$. Indeed, let $w_n=T_n(z_n)$ be a sequence in $\Omega_\infty$ that converges to $w=T(z)$, where $T_n, T \in \Gamma(\Omega)$ and $z_n,z \in \overline{\Omega}$. We have to show that $\tilde{f}(w_n) \to \tilde{f}(w)$. This is equivalent to $\tilde{f}((T^{-1} \circ T_n)(z_n)) \to \tilde{f}(z)$, since $\tilde{f}$ conjugates $\Gamma(\Omega)$ and $\Gamma(\Omega^*)$. We can therefore assume that $T$ is the identity, so that $w=z \in \overline{\Omega}$. Now, we show that every subsequence of $\tilde{f}(w_n)$ has a subsequence that converges to $\tilde{f}(w)$, which implies that $\tilde{f}(w_n) \to \tilde{f}(w)$, as desired. Consider an arbitrary subsequence, that we still denote by $\tilde{f}(w_n)$ for simplicity. If $T_n$ is equal to the identity for infinitely many $n$, then $\tilde{f}(w_n) = \tilde{f}(z_n)$ for these values of $n$, and this gives a subsequence that converges to $\tilde{f}(w)$, since $z_n \in \overline{\Omega}$ and $\tilde{f}$ is continuous on $\overline{\Omega}$. We can therefore assume, passing to a subsequence if necessary, that $T_n \neq \operatorname{id}$ for all $n$. Write $T_n = R_{j_n} \circ S_n$, where $R_{j_n}$ is the reflection in the disk $B_{j_n}$, and $S_n \in \Gamma(\Omega)$. Now, suppose that infinitely many $R_{j_n}$ are distinct. Then, passing to a subsequence, we get a sequence of pairwise disjoint disks $B_{j_n}$ with $w_n \in \br{B_{j_n}}$ for all $n$. The diameters of these disks converge to zero, so for each $n$ we can choose $w_n' \in \partial B_{j_n} \subset \overline{\Omega}$ with $|w_n' - w_n| \to 0$. In particular, the sequence $w_n'$ also converges to $w$, so that $\tilde{f}(w_n') \to \tilde{f}(w)$, by the continuity of $\tilde{f}$ on $\overline{\Omega}$. But by construction, both $\tilde{f}(w_n')$ and $\tilde{f}(w_n)$ belong to the closed disk $\overline{B_{j_n}^*}$, from which it follows that $|\tilde{f}(w_n) - \tilde{f}(w_n')| \to 0$. This implies that $\tilde{f}(w_n) \to \tilde{f}(w)$, as required.

Passing to a subsequence, we may therefore assume that $R_{j_n}=R_i$ for all $n$, for some $i \in \mathbb{N}$. Then $w_n = R_i(S_n(z_n)) \in \overline{B_i}$ for all $n$, and since $w_n \to w$ and $w \in \overline{\Omega}$, the only possibility is that $w \in \partial B_i$. In particular $R_i(w)=w$, and we get that $S_n(z_n) = R_i(w_n)$ converges to $w$. If $S_n$ is the identity for infinitely many $n$, or if $S_n= R_{k_n} \circ Q_n$ for infinitely many $n$ where the $R_{k_n}$ are distinct, then arguing as above gives $\tilde{f}(S_n(z_n)) \to \tilde{f}(w)$, after passing to a subsequence. Hence $R_i^* \circ \tilde{f} (S_n(z_n)) \to R_i^* \circ \tilde{f}(w)$, i.e., $\tilde{f}(w_n) \to \tilde{f}(w)$, again using the fact that $\tilde{f}$ conjugates the Schottky groups. The only remaining case is $S_n = R_k \circ Q_n$ for all $n$, for some $k \in \mathbb{N}$. But then each $w_n = R_i \circ R_k \circ Q_n (z_n)$ belongs to the closed disk $R_i(\overline{B_k})$, a compact subset of $B_i$. This is clearly impossible since $w_n \to w$ and $w \in \partial B_i$.

This completes the proof that $\tilde{f}$ is continuous on $\Omega_\infty$. The same argument shows that $\tilde{f}^{-1}$ is continuous on $\Omega_\infty^*$, and $\tilde{f}$ is a homeomorphism of $\Omega_\infty$ onto $\Omega_\infty^*$.

Now, we extend $\tilde{f}$ to the whole Riemann sphere in the following way. Suppose that $z \notin \Omega_\infty$. By Corollary \ref{corollary:OmegaInfinity}, the point $z$ belongs to a unique nested sequence of disks $D_k$ of the form $D_k\coloneqq  R_{i_1} \circ \dots \circ R_{i_k}(B_{i_{k+1}})$, where $i_j \neq i_{j+1}$ for all $j\in \N$. Each disk $D_k^*\coloneqq R_{i_1}^* \circ \dots \circ R_{i_k}^*(B_{i_{k+1}}^*)$ is contained in $\widehat{\C}\setminus \br{\Omega_k^*}$ by \eqref{equation:complement} and the disks $D_k^*$ are also nested by \eqref{equation:containment}. Corollary \ref{corollary:OmegaInfinity} implies that the sequence $D_k^*$ must shrink to a single point $w\in \RiemannSphere \setminus \Omega_\infty^*$. We define $\tilde{f}(z)\coloneqq w$. The existence and uniqueness of the sequence $D_k^*$ shrinking to $w$ from Corollary \ref{corollary:OmegaInfinity} implies that this definition gives a bijection $\tilde{f}$ from $\RiemannSphere \setminus \Omega_\infty$ onto $\RiemannSphere \setminus \Omega_\infty^*$.

Observe that for each disk $D_k$ of the form $R_{i_1} \circ \dots \circ R_{i_k}(B_{i_{k+1}})$ we have $\widetilde f( D_k)=D_k^*= R_{i_1}^* \circ \dots \circ R_{i_k}^*(B_{i_{k+1}}^*)$. Indeed, the equality $\widetilde f(D_k\cap \Omega_\infty)= D_k^*\cap \Omega_\infty^*$ follows from the definition of $\widetilde f$ on $\Omega_\infty$ and the fact that it conjugates the Schottky groups of $\Omega$ and $\Omega^*$. The equality $\widetilde f( D_k\setminus \Omega_\infty) =D_k^*\setminus \Omega_\infty^*$ follows from the definition of $\widetilde f$ in the previous paragraph, together with the nesting property \eqref{equation:containment}.

It remains to show that $\tilde{f}$ as defined is a homeomorphism of $\RiemannSphere$. By the compactness of $\RiemannSphere$, it suffices to show that $\tilde{f}$ is continuous. First, note that $\Omega_\infty$ is dense in $\RiemannSphere$, by Proposition \ref{disks:converge to 0}. Continuity will thus follow if we prove that if $z \in \RiemannSphere \setminus \Omega_\infty$ and if $z_n$ is a sequence in $\Omega_\infty$ converging to $z$, then $\tilde{f}(z_n) \to \tilde{f}(z)$. We fix a disk $D_k^*$ as above with $\tilde f(z)\in D_k^*$. Since $z_n\to z$, we have $z_n\in D_k$ for all sufficiently large $n$.  Hence, $\tilde f(z_n)\in D_k^*$ and $\limsup_{n\to\infty}|\tilde f(z_n)-\tilde f(z)|\leq \diam(D_k^*)$. Proposition \ref{disks:converge to 0} now concludes the proof. This completes the proof of the existence of a homeomorphic extension $\tilde{f}\colon \RiemannSphere \to \RiemannSphere$ of $f$ that conjugates the Schottky groups.
\\

Finally, to prove uniqueness, suppose that $g\colon \RiemannSphere \to \RiemannSphere$ is another homeomorphic extension of $f$ that conjugates the Schottky groups of $\Omega$ and $\Omega^*$. Then for $T \in \Gamma(\Omega)$ and $z \in \overline{\Omega}$, we have
$$\tilde{f}(T(z)) = T^*(\tilde{f}(z)) = T^*(g(z)) = g(T(z)),$$
since $\tilde{f}=f=g$ on $\overline{\Omega}$ and both $\tilde{f}$ and $g$ conjugate the Schottky groups. It follows that $\tilde{f} = g$ on $\Omega_\infty$, and hence everywhere on $\RiemannSphere$, by the continuity of $\tilde{f}$ and $g$ and by the density of $\Omega_\infty$.
\end{proof}

\subsection{Quasiconformality}
Our goal now is to show that the map $\tilde{f}$ of Lemma \ref{HomeoExtension} is actually $K$-quasi\-conformal on $\RiemannSphere$, for some $K$
depending only on $\Omega$. First, we give a brief introduction to quasiconformal mappings, {referring the reader to \cite{LEH} for more information.}

Let $K \geq 1$, let $U,V$ be domains in $\C$ and let $f\colon U \to V$ be an orientation-preserving homeomorphism. We say that $f$ is $K$-\textit{quasiconformal} on $U$ if it belongs to the Sobolev space $W_{\loc}^{1,2}(U)$ and satisfies the Beltrami equation
$$\partial_{\overline{z}}f=\mu \, \partial_z f$$
almost everywhere on $U$, for some measurable function $\mu \colon  U  \to \mathbb{D}$ with $\|\mu\|_\infty \leq \frac{K-1}{K+1}$. In this case, the number $K \geq 1$ is called the \textit{quasiconformal dilatation} of $f$ and the function $\mu$ its \textit{Beltrami coefficient}, denoted by $\mu_f$. An orientation-preserving homeomorphism of the Riemann sphere $\widehat \C$ is $K$-quasiconformal if it is $K$-quasiconformal in local coordinates, using the standard conformal charts of $\widehat \C$.

A mapping is conformal if and only if it is $1$-quasiconformal; see \cite[pp.~182--183]{LEH}. This is usually referred to as Weyl's lemma. Furthermore, inverses of $K$-quasiconformal mappings are also $K$-quasiconformal, and the composition of a $K_1$-quasiconformal mapping and a $K_2$-quasi\-conformal mapping is $K_1K_2$-quasi\-conformal; see \cite[p.~17]{LEH}. Another well-known property of quasiconformal mappings is that they preserve sets of area zero; see \cite[Theorem 1.3, p.~165]{LEH}.

The following theorem is of central importance in the theory of quasiconformal mappings; see \cite[Chapter V, p.~191 ff.]{LEH}.

\begin{theorem}[Measurable Riemann mapping theorem]
\label{MRMT}
Let $U$ be a domain in $\RiemannSphere$ and let $\mu\colon U \to \mathbb{D}$ be a measurable function with $\|\mu\|_\infty <1$. Then there exists a quasiconformal mapping $f$ on $U$ such that $\mu=\mu_f$, i.e.,
$$\partial_{\overline{z}}f=\mu \, \partial_z f$$
almost everywhere on $U$. Moreover, the map $f$ is unique up to postcomposition with a conformal map.
\end{theorem}

We can now state the main result of this section.

\begin{theorem}
\label{QCextension}
Let $\Omega$ be a circle domain containing $\infty$ and satisfying the quasihyperbolic condition. Then there exists a constant $K \geq 1$ depending only on $\Omega$ such that every conformal map $f$ of $\Omega$ onto another circle domain $\Omega^*$ with $f(\infty)=\infty \in \Omega^*$ extends to a unique $K$-quasiconformal mapping of $\RiemannSphere$ that conjugates the Schottky groups of $\Omega$ and $\Omega^*$, in the sense of Lemma \ref{HomeoExtension}.
\end{theorem}

The proof of Theorem \ref{QCextension} is based on conformal modulus estimates.

If $A \subset \widehat{\mathbb{C}}$ is a non-degenerate (open) topological annulus, {i.e., $\RiemannSphere \setminus A$ has precisely two connected components each containing more than one point,} then $A$ is conformally equivalent to a unique annulus of the form $A(z_0;1,r)= \{z: 1 < |z-z_0| < r\}$ for some $r>1$. Such an annulus whose boundary components are concentric circles is called a \textit{circular annulus}. We define the \textit{conformal modulus} of $A$ by
$$\operatorname{Mod}(A)\coloneqq\frac{1}{2\pi} \log{r}.$$
The conformal modulus $\Mod(A)$ is a conformal invariant: if two (non-degenerate) annuli $A$ and $A'$ are conformally equivalent, then they have the same conformal modulus.

We shall need the following properties of conformal modulus.

\begin{enumerate}[(M1)]
\item (Superadditivity (or Gr\"otzsch inequality)  \cite[Chapter I, Lemma 6.3]{LEH}) If $A_1, \dots, A_n$ are disjoint annuli contained in $A$ such that each annulus $A_j$ separates the complementary components of $A$ (in other words, the annuli $A_j$, $j=1,\dots,n$, are \textit{nested inside $A$}),  then
$$\operatorname{Mod}(A) \geq \sum_{j=1}^n \operatorname{Mod}(A_j).$$
\item (Teichm\"{u}ller's module theorem \cite[Chapter II, Section 1.3]{LEH}) If $A$ separates the points $0$ and $z_1$ from $z_2$ and $\infty$, then
$$\operatorname{Mod}(A) \leq 2 \mu \left( \sqrt{\frac{|z_1|}{|z_1|+|z_2|}} \right),$$
where $\mu(x)$ is a positive decreasing function of $x \in (0,1)$.
\end{enumerate}

It is well-known that the notion of conformal modulus of annuli is related to quasiconformality.

\begin{lemma}
\label{GeometricDefQC}
Let $f\colon \RiemannSphere \to \RiemannSphere$ be an orientation-preserving homeomorphism with $f(\infty)=\infty$. Suppose that there are positive constants $M_1,M_2$ such that for any circular annulus $A \subset \mathbb{C}$ with $\operatorname{Mod}(A) \geq M_1$ we have
\begin{equation}
\label{ModulusInequality}
\operatorname{Mod}(f(A)) \geq M_2 \operatorname{Mod}(A).
\end{equation}
Then $f$ is $K$-quasiconformal for some $K$ depending only on $M_1$ {and} $M_2$.
\end{lemma}

\begin{proof}
The following argument is sketched in \cite{SCH2}.

Fix $z_0 \in \mathbb{C}$, and assume without loss of generality that $f(z_0)=0$. For $\rho>0$ sufficiently small, let
$$R_\rho \coloneqq \max_{\theta} |f^{-1}(\rho e^{i\theta}) - z_0|$$
and
$$r_\rho \coloneqq \min_{\theta} |f^{-1}(\rho e^{i\theta})-z_0|,$$
so that the circular dilatation of $f^{-1}$ at $0$ is
$$H_{f^{-1}}(0) \coloneqq \limsup_{\rho \to 0} \frac{R_\rho}{r_\rho}.$$
If $A_\rho$ is the annulus $\{z: r_\rho < |z-z_0| < R_\rho \}$, then $f(A_\rho)$ is a topological annulus separating $0$ and $\infty$, and both boundary components of $f(A_\rho)$ intersect the circle $\{w:|w|=\rho\}$. By Teichm\"{u}ller's module theorem (M2), $\operatorname{Mod}(f(A_\rho))$ is uniformly bounded above, independently of $z_0$ and $\rho$. The assumption \eqref{ModulusInequality} then gives a uniform upper bound on $\operatorname{Mod}(A_\rho) = \frac{1}{2\pi}\log{(R_\rho/r_\rho)}$. It follows that the circular dilatation of $f^{-1}$ is uniformly bounded. By \cite[Chapter IV, Theorem 4.2]{LEH}, the map $f^{-1}$, and thus also $f$, is $K$-quasiconformal for some $K$ depending only on $M_1,M_2$.
\end{proof}

Now we return to the proof of Theorem \ref{QCextension}.

\begin{proof}[Proof of Theorem \ref{QCextension}]
Let $\Omega$ be a circle domain containing $\infty$ and satisfying the quasihyperbolic condition, and let $f$ be a conformal map of $\Omega$ onto another circle domain $\Omega^{*}$ with $f(\infty)=\infty \in \Omega^*$. Recall that by Lemma \ref{HomeoExtension}, the map $f$ extends to a unique homeomorphism of $\RiemannSphere$ that conjugates the Schottky groups of $\Omega$ and $\Omega^*$. For convenience, we denote this extension by the same letter $f$. We have to show that $f \colon \RiemannSphere \to \RiemannSphere$ is $K$-quasiconformal, for some constant $K \geq 1$ depending only on $\Omega$. Assuming the following  lemma, we finish the proof.
\begin{lemma}
\label{ModulusEstimate}
There exist positive constants $C_1$ and $C_2$ such that if $A \subset \mathbb{C}$ is any circular annulus with $\operatorname{Mod}(A) \geq C_1$, then
$\operatorname{Mod}(f(A)) \geq C_2.$
\end{lemma}

Let $L=e^{2\pi C_1}>1$, and let $A \subset \mathbb{C}$ be a circular annulus with
$$\operatorname{Mod}(A) > \frac{1}{2\pi}\log{L}=C_1.$$
Without loss of generality, assume that $A=\{z: 1<|z|<r\}$ where $r>L$. Let $n \in \mathbb{N}$ be such that $L^n < r \leq L^{n+1}$, and for $j \geq 1$, let $A_j\coloneqq\{z: L^{j-1}<|z|<L^{j}\}$, so that $\operatorname{Mod}(A_j)=C_1$. By Lemma \ref{ModulusEstimate} and property (M1) of conformal modulus, we have
\begin{eqnarray*}
\operatorname{Mod}(f(A)) &\geq& \sum_{j=1}^{n} \operatorname{Mod}(f(A_j))\\
&\geq& nC_2\\
&=&  \frac{n}{n+1} \frac{2\pi C_2}{\log{L}} \left(\frac{1}{2\pi} (n+1)\log{L}\right)\\
&\geq& \frac{1}{2} \frac{2\pi C_2}{\log{L}} \left( \frac{1}{2\pi} \log{r}\right) \\
&\eqqcolon& M_2 \operatorname{Mod}(A).
\end{eqnarray*}
The result now follows from Lemma \ref{GeometricDefQC}.
\end{proof}

\subsection{Proof of Lemma \ref{ModulusEstimate}}
The proof will be very similar to the proofs of Lemma \ref{lemma:circles to circles} and Lemma \ref{lemma:points}, but here we shall need to make use of the absolute continuity lemmas from Subsection \ref{ac:section}. This is a complication that is not present in the proof of He and Schramm \cite[Lemma 4.2]{SCH2}, because of their assumption that the boundary of the domain has $\sigma$-finite length.

We fix a circular annulus $A$ of modulus greater than $\frac{1}{2\pi} \log{2}$ so that it contains the closure of a circular annulus $A_0$ that is homothetic to the annulus $A(0;1,2)$. Using a homothety, we may assume that $A_0=A(0;1,2) \subset \subset A$. Let $A^*=f(A)$, and consider a conformal map $h$ from $A^*$ to a circular annulus $A^\#=A(0;1,R)$. It suffices to bound the modulus of $A^\#$, i.e., $\frac{1}{2\pi} \log R$, from below.

Consider the sets $\overline{\Omega_k}$ of Subsection \ref{subsechom}. The set $\widehat{\C}\setminus \overline{\Omega_k}$ consists of (at most) countably many open disks (by \eqref{equation:complement}) with diameters converging to $0$ as $k\to\infty$, by Proposition \ref{disks:converge to 0}. Hence, if $k$ is sufficiently large, we may {assume} that each such complementary disk $B$ intersecting the annulus $A_0$ is so small that it is contained in $A$, and has distance at least $2\diam(B)>0$ from $\partial A$. Since $f$ is a homeomorphism, the disk $f(B)$ also has a distance $2\diam(f(B))>0$ from $\partial A^*$, whenever $B\cap A_0\neq \emptyset$, provided that we pick an even larger $k$. Property (F4) in Subsection \ref{sec:fatness} now shows that $h(f(B))$ is $c$-fat for a universal $c$. Moreover, we may have that $\diam(h(f(B))) \leq 1/2$, if $k$ is sufficiently large.

Summarizing, if we consider the map $g=h\circ f\colon  A\to A^\#$, we have that each complementary disk $B\subset \widehat{\C}\setminus \overline{\Omega_k}$ intersecting $A_0\subset\subset A$ is contained in $A$ and its image $g(B)$ is $c$-fat and has diameter bounded above by $1/2$. Here, $k$ is fixed, and is sufficiently large, depending on $A, A_0$, and $f$. We  denote by $H$ the collection of these disks $B$.

By Proposition \ref{ac:lemma Sobolev} we have that $\mathcal H^1(f(\gamma_r\cap \partial \Omega))=0$ for a.e.\ $1<r<2$, where $\gamma_r\subset A_0$ denotes the circle of radius $r$ around $0$. We claim that this remains valid if $\partial \Omega$ is replaced with $\partial \Omega_k$. Indeed, reflecting across a circle component of $\partial \Omega$ yields a domain $R(\Omega)$ inside this circle that is bi-Lipschitz-equivalent to $\Omega$, away from $\infty$. In general, if $T\in \Gamma(\Omega)$, then $T$ is the composition of finitely many reflections, so $T(\Omega)$ is bi-Lipschitz-equivalent to $\Omega$ away from $\infty$. By Remark \ref{quasi:invariant}, the quasihyperbolic condition \eqref{quasi:condition} is invariant under bi-Lipschitz maps, so the reflected domain $T(\Omega)$ also satisfies it. It follows from Proposition \ref{ac:lemma Sobolev} that $\mathcal H^1(f(\gamma_r\cap \partial T(\Omega)))=H^1(f(\gamma_r\cap  T(\partial \Omega)))=0$ for a.e.\ $1<r<2$. Note that there are at most countably many $T\in \Gamma(\Omega)$ with $l(T)\leq k$. Therefore, by \eqref{equation:boundaries} we have
\begin{align*}
\mathcal H^1(f(\gamma_r\cap \partial \Omega_k)) = \mathcal H^1\Bigg(f\bigg(\gamma_r\cap \bigcup_{l(T)\leq k}T(\partial \Omega )\bigg)\Bigg) \leq \sum_{l(T)\leq k} \mathcal H^1(f(\gamma_r\cap T(\partial \Omega)))=0
\end{align*}
for a.e.\ $1<r<2$. Since the conformal map $h\colon A^*\to A^\#$ is smooth, it is Lipschitz continuous on $f(\gamma_r)\subset f(A_0)\subset \subset A^*$, so it follows that
\begin{align}\label{estimate:H1}
\mathcal H^1( g(\gamma_r\cap \partial \Omega_k))=0
\end{align}
for a.e.\ $1<r<2$.

We fix such an $r\in (1,2)$. Using this absolute continuity property we wish to apply Lemma \ref{ac:lemma Real} in order to prove the estimate
\begin{align}\label{estimate:g}
1 \leq \int_{\gamma_r \cap  \Omega_k} |g'|\, ds + \sum_{\substack{B\in H\\ B\cap \gamma_r\neq \emptyset}} \diam(g(B)).
\end{align}
Consider the function $g\circ \gamma_r(x)= g(re^{ix})$, $x\in  \R$, and let $Z=  \gamma_r^{-1}(\br{\Omega_k})$ and $K=\gamma_r^{-1}(\partial \Omega_k)$; here $\gamma_r$ is treated as a complex-valued function of a real variable, rather than as a set. Note that $\partial Z\subset \gamma_r^{-1}(\partial {\br {\Omega_k}})$; this follows from the general fact that if $\phi\colon X\to Y$ is a continuous map between topological spaces and $C\subset Y$ is a closed set, then $\partial \phi^{-1}(C)\subset \phi^{-1}(\partial C)$. We trivially have $\partial \br{\Omega_k} \subset \partial \Omega_k$, so $\partial Z$ is contained in $K$, as required in the statement of Lemma \ref{ac:lemma Real}. Consider now the function $G\coloneqq (g\circ \gamma_r)|_{Z}$ and observe that  $G$ is locally absolutely continuous on $Z\setminus K =  \gamma_r^{-1}(\Omega_k)$; indeed if $\gamma_r(x)=re^{ix}\in \Omega_k$, then $g=h\circ f$ is conformal in a neighborhood of $re^{ix}$, so $G'$ exists and is continuous in a neighborhood of $x$. Finally, by \eqref{estimate:H1} we have
$$\mathcal H^1(G(K))= \mathcal H^1( g(\gamma_r\cap \partial \Omega_k))=0.$$
We are exactly in the setting of Lemma \ref{ac:lemma Real}, so if $(x_i,y_i)$, $i\in I$, denote the complementary intervals of $Z$, then
\begin{align*}
|G(x)-G(y)| \leq  \int_{[x,y]\cap (Z\setminus K)} |G'(x)|dx + \sum_{[x,y]\cap (x_i,y_i)\neq \emptyset}|G(x_i)-G(y_i)|
\end{align*}
for all $x,y\in Z$. Since $\widehat{\C} \setminus \br{\Omega_k}$ is the union of disjoint open disks by \eqref{equation:complement}, each complementary interval $(x_i,y_i)$ of $Z$ is mapped by $\gamma_r$ to an arc contained in a disk $B\subset \widehat{\C} \setminus \br {\Omega_k}$ and whose endpoints lie in $\partial B$. Moreover, if $x_0\in Z$, then the complementary intervals $(x_i,y_i)$ of $Z$ contained in $[x_0,x_0+2\pi]$ correspond to distinct disks $B\subset \widehat{\C} \setminus \br {\Omega_k}$, because the intersection of an open disk $B$ with the circle $\gamma_r([x_0,x_0+2\pi])$ cannot consist of more than one arcs. Hence, the above estimate implies that
\begin{align*}
|G(x)-G(y)|\leq \int_{\gamma_r \cap  \Omega_k} |g'|\, ds + \sum_{\substack{B\in H\\ B\cap \gamma_r\neq \emptyset}} \diam(g(B))
\end{align*}
for all $x,y\in [x_0,x_0+2\pi]$. If we show that there exists $y_0\in [x_0,x_0+2\pi]$ with $|G(x_0)-G(y_0)|\geq 1$, then we arrive to the desired inequality \eqref{estimate:g}.

The point $G(x_0)\in A^{\#}$ lies outside the ball $B(0,1)$. Since $g(\gamma_r)$ surrounds the ball $B(0,1)$, there exists a point $w_0\in g(\gamma_r)$ that is ``antipodal" to $G(x_0)$, in the sense that it lies on the line through $0$ and $G(x_0)$ and $|G(x_0)-w_0|\geq 2$. If $w_0\in G(Z)$, then there exists $y_0\in [x_0,x_0+2\pi]$ such that $G(y_0)=w_0$, so our claim is proved. If $w_0\notin G(Z)$, then $w_0\in g(B)$ for some disk $B\in H$ intersecting $\gamma_r$. Then there exists a complementary interval $(x_i,y_i)\subset [x_0,x_0+2\pi]$ of $Z$ with $\gamma_r(x_i),\gamma_r(y_i)\in \partial B$, so $G(y_i)\in \partial g(B)$. We have $|w_0-G(y_i)| \leq \diam(g(B))\leq 1/2$. Therefore $|G(x_0)-G(y_i)|\geq |G(x_0)-w_0|-|w_0-G(y_i)| \geq 1$, and our claim is proved with $y_0=y_i$.

Now we proceed exactly as in the the proof of \cite[Lemma 4.2]{SCH2}. We integrate \eqref{estimate:g} over $r\in (1,2)$ to get
\begin{align}\label{ModulusEstimate:ineq}
1\leq \int_{A_0\cap \Omega_k} |g'|  + \sum_{B\in H} \diam(g(B)) d(B),
\end{align}
where $d(B)=\mathcal H^1( \{s\in [1,2]: B\cap B(0,s)\neq \emptyset \} )$, and $d(B)^2\lesssim \area(B\cap A_0)$ by (F1). By the Cauchy-Schwarz inequality, the square of the integral term is bounded by
\begin{align*}
\int_{A_0\cap \Omega_k} |g'|^2 \cdot \area(A_0\cap \Omega_k) \leq \area(A^\#)\cdot \area(A_0) =  \pi(R^2-1)\cdot 3\pi.
\end{align*}
Similarly, using the fatness of $g(B)$, the square of the summation term in \eqref{ModulusEstimate:ineq} can be bounded by
\begin{align*}
\sum_{B\in H} \diam(g(B))^2 \cdot \sum_{B\in H} d(B)^2 &\lesssim \sum_{B\in H} \area(g(B)) \cdot \sum_{B\in H} \area(B\cap A_0)\\
&\lesssim \area(A^\#)\cdot \area(A_0) = 3\pi^2(R^2-1).
\end{align*}
Summarizing, by \eqref{ModulusEstimate:ineq} we have $R^2-1\geq C$ for a universal $C>0$, which implies that $\log R$ is bounded from below, as desired.\qed

\section{Proof of Theorem \ref{mainthm1}}
\label{sec8}

In this section, we conclude the proof of Theorem \ref{mainthm1}. Let $\Omega$ be a circle domain containing $\infty$ and satisfying the quasihyperbolic condition, and let $f\colon  \Omega \to \Omega^*$ be a conformal map of $\Omega$ onto another circle domain $\Omega^*$. Without loss of generality, assume that $f(\infty)=\infty \in \Omega^*$. By Theorem \ref{QCextension}, the map $f$ extends to a $K$-quasiconformal mapping of the whole sphere that conjugates the Schottky groups, for some $K \geq 1$ depending only on $\Omega$. Again, we denote the extension by the same letter $f$.

We now use quasiconformal deformation of Schottky groups to prove that $f$ must be a M\"{o}bius transformation, thereby showing that $\Omega$ is conformally rigid. First, we need some preliminaries on invariant Beltrami coefficients with respect to a Schottky group.

Let $V \subset \mathbb{C}$ be open and let $\mu\colon V \to \mathbb{D}$ be measurable. If $f \colon  U \to V$ is a quasiconformal mapping, then we define a measurable function $f^{*}\mu\colon U \to \mathbb{D}$, called the \textit{pullback} of $\mu$ under $f$, by
$$f^{*}\mu \coloneqq  \frac{\partial_{\overline{z}}f + (\mu \circ f) \overline{\partial_z f}}{\partial_z f + (\mu \circ f)\overline{\partial_{\overline{z}}f}}.$$
Similarly, we define the Beltrami coefficient of an orientation-reversing quasiconformal mapping $f\colon U \to V$ by $\mu_f \coloneqq  \mu_{\overline{f}}$ and the pullback $f^{*}\mu$ by
$$f^{*}\mu = \frac{\overline{\partial_{z}f} + \overline{(\mu \circ f)} \partial_{\overline{z}} f}{\overline{\partial_{\overline{z}} f} + \overline{(\mu \circ f)}\partial_{z}f}.$$
Here by \textit{orientation-reversing} quasiconformal mapping we mean the complex conjugate of a quasiconformal mapping. The notation $f^*$ here should not be confused with the notation in Proposition \ref{proposition:topological}.

With these definitions, the coefficient $\mu_f$ is simply the pullback of $\mu_0 \equiv 0$ under $f$. Moreover, pullbacks satisfy the natural property
$$(f \circ g)^{*}\mu = g^*(f^*\mu).$$

Now, we say that a measurable function $\mu\colon \RiemannSphere \to \mathbb{D}$ is \textit{invariant with respect to} a Schottky group $\Gamma(\Omega)$ if
$T^*\mu=\mu$ almost everywhere on $\RiemannSphere$, for every $T \in \Gamma(\Omega)$. This is equivalent to
$$\mu = (\mu \circ T) \overline{\partial_z T} / \partial_z T$$
or
$$\mu = \overline{(\mu \circ T)} \partial_{\overline{z}} T / \overline{\partial_{\overline{z}}T}$$
depending on whether $T$ is M\"{o}bius or anti-M\"{o}bius, respectively.

\begin{proposition}
\label{PropInvariant}
Let $f\colon \RiemannSphere \to \RiemannSphere$ be a quasiconformal mapping and let $\Omega$ be a circle domain. Then the Beltrami coefficient $\mu_f$ is invariant with respect to the Schottky group $\Gamma(\Omega)$ if and only if $f$ maps $\Omega$ onto another circle domain $\Omega^*$ and it conjugates $\Gamma(\Omega)$ and $\Gamma(\Omega^*)$.
\end{proposition}

\begin{proof}
Suppose that $\mu_f$ is invariant with respect to $\Gamma(\Omega)$. To prove that $\Omega^*=f(\Omega)$ is a circle domain, it suffices to show that $f(\gamma_j)$ is a circle for each circle $\gamma_j$ in $\partial \Omega$. Fix $\gamma_j$, and as before let $R_j$ be the reflection across $\gamma_j$. We have

$$(f \circ R_j \circ f^{-1})^{*}\mu_0 = (f^{-1})^*(R_j^*(f^{*}\mu_0)) = (f^{-1})^*(f^{*}\mu_0)=\mu_0,$$
where we used the fact that $R_j^{*}\mu_f = \mu_f$, by invariance of $\mu_f$. This shows that the Beltrami coefficient of the map $f \circ R_j \circ f^{-1}$ is zero almost everywhere, so that it is anti-M\"{o}bius. But it fixes $f(\gamma_j)$, so that $f(\gamma_j)$ must be a circle and ${R_j^*\coloneqq }f \circ R_j \circ f^{-1}$ is the reflection across that circle. It follows that $\Omega^*$ is a circle domain, as required. Moreover, we clearly have $$T^* = f \circ T \circ  f^{-1} \qquad (T \in \Gamma(\Omega)),$$ where $T^*$ denotes the element of $\Gamma(\Omega^*)$ corresponding to $T\in \Gamma(\Omega)$ (see also the comments before Lemma \ref{HomeoExtension}), so that $f$ indeed conjugates $\Gamma(\Omega)$ and $\Gamma(\Omega^*)$.

Conversely, suppose that $f$ maps $\Omega$ onto another circle domain $\Omega^*$, and that it conjugates $\Gamma(\Omega)$ and $\Gamma(\Omega^*)$. Then for each $j$, the map $f \circ R_j \circ f^{-1}$ is anti-M\"{o}bius, so that
$$\mu_0 = (f \circ R_j \circ f^{-1})^{*}\mu_0 = (f^{-1})^*(R_j^*(f^{*}\mu_0)).$$
Taking the pullback of both sides by $f$ gives $\mu_f = R_j^{*}(\mu_f)$. This {implies} that $\mu_f$ is invariant with respect to $\Gamma(\Omega)$.
\end{proof}

We can now proceed with the proof of Theorem \ref{mainthm1}.

\begin{proof}
Recall that by Theorem \ref{QCextension}, every conformal map of $\Omega$ onto another circle domain extends to a unique $K$-quasiconformal mapping of $\RiemannSphere$ that conjugates the Schottky groups, for some $K \geq 1$ depending only on $\Omega$.

Let $f$ be such a map, and suppose that $f$ is not M\"{o}bius, so that $K>1$ and $\|\mu_f\|_{\infty}>0$. Let $(K-1)/(K+1) < c <1$, and set $\nu\coloneqq  (c/\|\mu_f\|_\infty)\mu_f$. By Proposition \ref{PropInvariant}, the coefficient $\mu_f$ is invariant with respect to $\Gamma(\Omega)$, and thus so is $\nu$. By the measurable Riemann mapping theorem (Theorem \ref{MRMT}), there is a quasiconformal mapping $h$ of $\RiemannSphere$ with $\mu_h = \nu$ and $h$ conformal on $\Omega$. Again by Proposition \ref{PropInvariant}, we have that $h$ maps $\Omega$ onto another circle domain and it conjugates the Schottky groups $\Gamma(\Omega)$ and $\Gamma(h(\Omega))$.

Now, again by Theorem \ref{QCextension}, the restriction of $h$ to $\Omega$ has a unique $K$-quasi\-conformal extension $\tilde{h}$ to the whole sphere that also conjugates $\Gamma(\Omega)$ and $\Gamma(h(\Omega))$. By uniqueness, it follows that $h=\tilde{h}$ everywhere on $\RiemannSphere$. But this contradicts the fact that $h$ is not $K$-quasiconformal, since $\|\mu_h\|_{\infty}=\|\nu\|_\infty= c>(K-1)/(K+1)$.
\end{proof}

\begin{remark}
The above proof relies on the fact that the {quasiconformal dilatation} $K$ in Theorem \ref{QCextension} is uniform, in the sense that it does not depend on the conformal map. We mention though that in order to prove that $\Omega$ is rigid, it actually suffices to prove that every conformal map of $\Omega$ onto another circle domain extends to a quasiconformal mapping of the whole sphere (regardless of the quasiconformal dilatation). This follows from \cite[Theorem 5]{YOU2}.

\end{remark}

\section{Further remarks on the rigidity conjecture}
\label{sec9}

Recall from the introduction that the rigidity conjecture states that a circle domain $\Omega$ is conformally rigid if and only if its boundary is conformally removable. If there are only point boundary components, then the rigidity of $\Omega$ clearly implies the removability of $\partial \Omega$. Whether the converse holds remains unknown even in this special case. The goal of this section is to investigate the properties of a possible counterexample.

Suppose that $\Omega$ is a circle domain having only point boundary components and assume that $\partial \Omega$ is removable but $\Omega$ is not rigid. Then there exists a non-M\"{o}bius conformal map $f\colon  \Omega \to \Omega^*$ with $f(\infty)=\infty \in \Omega^*$, where $\Omega^*$ is another circle domain. In particular, the domain $\Omega^*$ is also non-rigid.

Now, we note that $\partial \Omega^*$ necessarily contains at least one circle. Indeed, if not, then $f$ would be a conformal map between the complements of two totally disconnected compact sets and thus would extend to a homeomorphism of $\RiemannSphere$ by Lemma \ref{lemma:extension}. But then $f$ would be a M\"{o}bius transformation, by the removability of $\partial \Omega$.

We mention that in \cite[Theorem 4.1]{GEM}, based on results of Ahlfors and Beurling \cite{AHB}, Gehring and Martio construct a circle domain $\Omega$ having only point boundary components and a conformal map of $\Omega$ onto another circle domain $\Omega^{*}$ having exactly one boundary circle. In their example, however, the boundary of $\Omega$ has positive area and hence is not removable. The question is whether one can construct such an example but with removable boundary.

The following result states that if $\partial \Omega$ is removable, then $\partial \Omega^{*}$ must in fact contain a lot of circles.

\begin{theorem}
\label{CounterExample}
Suppose that $f\colon \Omega \to \Omega^{*}$ is a non-M\"{o}bius conformal map between two circle domains with $f(\infty)=\infty \in \Omega^{*}$, and assume that $\Omega$ has only point boundary components. If $\partial \Omega$ is removable, then $\Omega^{*}$ has the property that every point $w_0 \in \partial \Omega^{*}$ that is not a point boundary component is the accumulation point of an infinite sequence of distinct circles in $\partial \Omega^{*}$.
\end{theorem}

A circle domain $\Omega^{*}$ with this property is called a \textit{Sierpi\'nski-type circle domain}. See Figure \ref{fig:sierpinski}. The proof of Theorem \ref{CounterExample} is closely related to the notion of local removability.

\begin{figure}[h!t!b]
\label{fig1}
\begin{center}
\includegraphics[width=6cm, height=6cm]{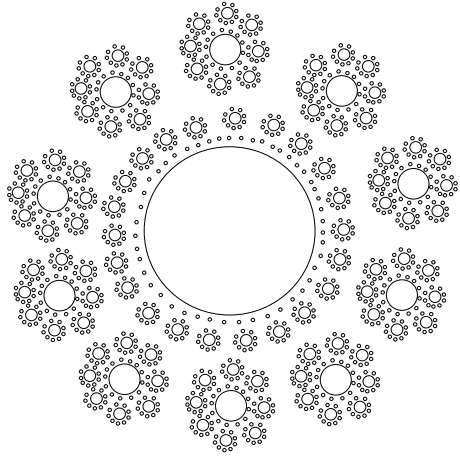}
  \caption{A Sierpi\'nski-type circle domain.}\label{fig:sierpinski}
\end{center}
\end{figure}

\begin{definition}
\label{def1}
A compact set $E \subset \mathbb{C}$ is \textit{locally conformally removable} if for any open set $U$, every homeomorphism on $U$ that is conformal on $U \setminus E$ is actually conformal on the whole open set $U$.
\end{definition}

Note that if $E$ is conformally removable, then for any open set $U$ with $E \subset U$, every homeomorphism on $U$ that is conformal on $U \setminus E$ is actually conformal everywhere on $U$ (see \cite[Proposition 11]{YOU2}). The main difference here is that $E$ is not assumed to be contained in $U$; in particular, the set $E$ may intersect $\partial U$. Clearly, local removability implies removability. Whether the converse holds remains unknown. We state this as a conjecture; see also \cite[Question 4]{Bishop}.

\begin{conjecture}
\label{LocalRemovability}
A compact set is conformally removable if and only if it is locally conformally removable.
\end{conjecture}

Conjecture \ref{LocalRemovability}, if true, would imply that the union of two conformally removable sets is also conformally removable, which is an open problem. In \cite{YOU2}, the conjecture was shown to hold for various sets, such as quasiarcs and totally disconnected compact sets.

We now proceed with the proof of Theorem \ref{CounterExample}.

\begin{proof}
Let $f\colon \Omega \to \Omega^*$ be a non-M\"{o}bius conformal map between two circle domains with $f(\infty)=\infty \in \Omega^*$, where $\partial \Omega$ is a totally disconnected removable compact set.

Suppose for a contradiction that there exists a complex number $w_0 \in \partial \Omega^*$ belonging to a boundary circle $\gamma^*$ and an open set $V$ containing $w_0$ that is disjoint from all the other circles in $\partial \Omega^*$. Shrinking $V$ if necessary, we can assume that $V$ is a disk centered at $w_0$ sufficiently small so that it does not contain $\gamma^*$. Also, let $z_0 \in \partial \Omega$ be the point corresponding to $\gamma^*$ under $f$, i.e., $f^*(\{z_0\})=\gamma^*$; see Proposition \ref{proposition:topological}.

Now, let $b^*$ be a component of $ \partial \Omega^*$ contained in $V$. Note that $b^*$ is a single point, and denote by $b$ the point component of $\partial \Omega$ corresponding to $b^*$ under $f$, i.e., $f^*(b)=b^*$. By (PT1) of Section \ref{sec3}, for each $\varepsilon>0$, we can find a Jordan curve $\gamma$ in $B(b,\varepsilon) \cap \Omega$ whose bounded complementary component $U_1$ contains $b$. Then $f(\gamma)$ is a Jordan curve in $\Omega^*$ whose bounded complementary component $V_1$ contains $b^*$, by Lemma \ref{lemma:surround}. Moreover, the curve $f(\gamma)$ is contained in $V$ provided $\varepsilon$ is small enough. Now, note that the sets $E\coloneqq U_1 \cap \partial \Omega$ and $F\coloneqq  V_1 \cap \partial \Omega^*$ are compact and totally disconnected, and that $f\colon U_1 \setminus E \to V_1 \setminus F$ is conformal. It follows that $f$ extends to a homeomorphism of $U_1$ onto $V_1$; see Lemma \ref{lemma:extension}. But $E$ is removable, as a compact subset of the removable set $\partial \Omega$. We thus obtain from the remark after Definition \ref{def1} that $f$ extends to be conformal on $U_1$, a neighborhood of $b$.

Summarizing, we have proved that $f^{-1}$ extends to be conformal in a neighborhood of every point $b^* \in V \cap \partial \Omega^*$. It follows that the map $f^{-1}$ has a conformal extension to the open set $V \setminus D^*$, where $D^*$ is the closed disk bounded by the circle $\gamma^*$. Moreover, we have that $f^{-1}(w_n)$ converges to $z_0$ whenever $w_n$ is a sequence in $V \setminus D^*$ accumulating at the circular arc $\gamma^* \cap V$. This is clearly impossible, by the Schwarz reflection principle for example, and we get a contradiction.
\end{proof}

Theorem \ref{CounterExample} shows the importance of studying the rigidity of Sierpi\'nski-type circle domains. It is worth mentioning that such domains also appear naturally in another conjecture by He and Schramm on the removability of the boundaries of circle domains.

\begin{conjecture}[Removability conjecture \cite{SCH2}]
\label{RemovabilityConjecture}
Let $\Omega$ be a circle domain. If every Cantor set contained in $\partial \Omega$ is conformally removable, then $\partial \Omega$ is removable.
\end{conjecture}
The related question whether a non-removable curve contains a non-removable Cantor set  was asked by Bishop in \cite[Question 4]{Bishop}.

In \cite{YOU2}, Conjecture \ref{RemovabilityConjecture} was shown to hold whenever the set of accumulation points of circles is not too large, in some precise sense. Sierpi\'nski-type circle domains are therefore good candidates for a counterexample.

Finally, we conclude by mentioning that it is possible to construct a Sierpi\'nski-type circle domain $\Omega$ such that
\begin{enumerate}
\item $\partial \Omega$ does not have $\sigma$-finite length,
\item $\Omega$ is a John domain.
\end{enumerate}
In particular, this gives examples of rigid circle domains (by Theorem \ref{mainthm1}) for which the rigidity result from \cite{SCH2} does not apply.

\bibliographystyle{amsplain}

\end{document}